\pdfoutput=1 
\documentclass[11pt]{article}
\usepackage[left=1in,top=1in,right=1in,bottom=1in,letterpaper]{geometry}

\usepackage[utf8]{inputenc} 
\usepackage[T1]{fontenc}    
\usepackage[colorlinks,citecolor=blue,urlcolor=blue,pagebackref]{hyperref}       
\hypersetup{
  colorlinks,
  linkcolor={blue!50!black},
  citecolor={blue!50!black},
  urlcolor={blue!80!black}
}
\usepackage{url}            
\usepackage{booktabs}       
\usepackage{amsfonts}       
\usepackage{nicefrac}       
\usepackage{microtype}      
\usepackage{xcolor}         
\usepackage{color}
\usepackage{bm,multirow,xspace}
\usepackage{enumerate,mathtools}
\usepackage{graphicx}
\usepackage{epstopdf,amsmath}
\usepackage{bbm}

\usepackage{amssymb}
\usepackage{mathrsfs,amsthm}
\usepackage{makecell}

\usepackage{subfigure}

\allowdisplaybreaks

\usepackage{amsmath}    
\usepackage{algorithm}
\usepackage{algpseudocode}
\usepackage{colortbl}
\usepackage[inline]{enumitem}

\usepackage{amsmath,color}

\usepackage{algpseudocode}
\usepackage{graphicx}
\usepackage{epstopdf}
\usepackage{pifont}

\usepackage{multirow}

\newcommand{\xmark}{\ding{55}}%
\allowdisplaybreaks
\newtheorem{definition}{Definition}

\DeclareMathOperator*{\argmax}{argmax}
\DeclareMathOperator*{\argmin}{argmin}
\DeclareMathOperator{\sign}{sign}
\newcommand{\abs}[1]{\left\lvert #1 \right\rvert}
\DeclareMathOperator*{\iid}{\texttt{iid}}
\newcommand{\norm}[1]{\left\lVert #1 \right\rVert}
\newcommand{\ceil}[1]{\lceil #1 \rceil}

\newcommand{\expec}[2]{\mathbb{E}_{#2}\left[ #1 \right] }
\newcommand\numberthis{\addtocounter{equation}{1}\tag{\theequation}}  

\def\fn[#1]#2{{f_{#1}\left(x_{#2}\right)}}

\newtheorem{lemma}{Lemma}[section]
\newtheorem{theorem}{Theorem}[section]

\newtheorem{proposition}{Proposition}[section]
\newtheorem{assumption}{Assumption}[section]
\newtheorem{remark}{Remark}

\newcommand{\inprod}[1]{\left\langle {#1}\right\rangle}

\newcommand{\kb}[1]{{\color{magenta}\bf[KB: #1]}}

\def\L{{Lipschitz }}

\def\exp{{\rm exp}}

\def\cA{{\cal A}}

\def\cB{{\cal B}}

\def\cD{{\cal D}}

\def\cW{{\cal W}}

\def\cI{{\cal I}}
\def\cF{{\cal F}}

\def\tt{{\tilde \theta}}
\def\ty{{\tilde y}}
\def\tz{{\tilde z}}

\newenvironment{talign*}
 {\csname align*\endcsname}
 {\endalign}

\title{Constrained Stochastic Nonconvex Optimization with State-dependent Markov Data}

\author{Abhishek Roy\thanks{Halıcıoğlu Data Science Institute, University of California, San Diego. \texttt{abroy@ucdavis.edu}.} 
\and Krishnakumar Balasubramanian\thanks{Department of Statistics, University of California, Davis. \texttt{kbala@ucdavis.edu}. } 
\and Saeed Ghadimi\thanks{Department of Management Sciences, University of Waterloo. \texttt{sghadimi@uwaterloo.ca}.}
}

\begin{document}

\maketitle

\begin{abstract}
We study stochastic optimization algorithms for constrained nonconvex stochastic optimization problems with Markovian data. In particular, we focus on the case when the transition kernel of the Markov chain is state-dependent. Such stochastic optimization problems arise in various machine learning problems including strategic classification and reinforcement learning. For this problem, we study both projection-based and projection-free algorithms. In both cases, we establish that the number of calls to the stochastic first-order oracle to obtain an appropriately defined $\epsilon$-stationary point is of the order $\mathcal{O}(1/\epsilon^{2.5})$. In the projection-free setting we additionally establish that the number of calls to the linear minimization oracle is of order $\mathcal{O}(1/\epsilon^{5.5})$. We also empirically demonstrate the performance of our algorithm on the problem of strategic classification with neural networks.
\end{abstract}

\section{Introduction}
We consider the following stochastic optimization problem 
\begin{align}\label{eq:mainproblem}
    \argmin_{\theta\in\Theta}f(\theta)=\argmin_{\theta\in\Theta}\expec{F(\theta;x)}{},
\end{align}
where (\textrm{i}) the expectation is taken over the stationary distribution, $\pi_\theta$, of the random vector $x$, (\textrm{ii}) $F$ (and hence $f$) is a potentially non-convex function in $\theta$, and (\textrm{iii}) $\Theta$ is a compact and convex constraint set. Stochastic approximation algorithms for solving problem~\eqref{eq:mainproblem}, given an independent and identically distributed ($\iid$) data stream $\{x_k\}_k$ drawn from a distribution $\pi$, are well-studied. Such $\iid$ assumptions are commonly made in various machine learning and statistical problems including empirical risk minimization~\cite{shalev2014understanding}, sparse recovery~\cite{bach2012structured} and compressed sensing~\cite{foucart2013invitation,lan2020first}. We refer to~\cite{moulines2011non, agarwal2012information, rakhlin2012making, ghadimi2013stochastic,shamir2013stochastic, lan2016conditional, arjevani2019lower} for a partial list of non-asymptotic upper and lower bounds on the oracle complexity of widely-used stochastic approximation algorithms like the Stochastic Gradient Descent (SGD) and the Stochastic Conditional Gradient Algorithm, in the $\iid$ setting.

Our focus in this work is on the case when the data sequence $\{x_k\}_k$ is not necessarily $\iid$. Such non-$\iid$ settings arise frequently in several machine learning applications including but not limited to strategic classification~\cite{hardt2016strategic,cai2015optimum,mendler2020stochastic,li2022state} and reinforcement learning~\cite{bartlett1992learning, goldberg2013adaptive, zhang2021statistical,karimi2019non, qu2020finite}. While sample average approximation (or empirical risk minimization) with non-$\iid$ data is relatively well-understood (see, for example,~\cite{agarwal2012generalization,kuznetsov2017generalization, dagan2019learning, roy2021empirical} and references therein), a deeper understanding of the non-asymptotic oracle complexity of stochastic approximation algorithms for non-$\iid$ data is only now starting to emerge. 

Towards that, we establish non-asymptotic oracle complexity results for the stochastic conditional gradient algorithm for non-convex constrained stochastic optimization with Markovian data with potentially state-\emph{dependent} transition kernels. To establish our results, from a methodological point-of-view, we leverage the moving-average stochastic gradient estimation technique recently used in~\cite{zhang2020one,ghadimi2020single,xiao2022projection} in the context of constrained optimization with $\iid$ data. This technique avoids having to a use a mini-batch of samples in each iteration, which turns out to be crucial in the non-$\iid$ setup we consider. From a theoretical point-of-view, we assume the so-called drift conditions, a classical assumption in Markov Chain literature \cite{andrieu2005stability}. This ensures the existence of a solution to the Poisson equation associated with the underlying Markov chain~\cite{douc2018markov} which enables one to decompose the noise present in the stochastic gradient into three components: a martingale difference sequence, a time-decaying sequence, and a telescopic sum type sequence. The key idea of our paper is to use this decomposition to construct an auxiliary sequence of iterates with a time-decaying noise variance and show that these sequence of iterates are \textit{close} to the iterates of the original sequence produced by our algorithm. This novel technique in then used in combination with a merit-function based analysis to establish the oracle complexity results.
\subsection{Motivating Example}\label{sec:motivate}
Problems of the form in \eqref{eq:mainproblem} arise in various important applications, e.g., strategic classification, and reinforcement learning as mentioned above. Below we illustrate the motivation of this work through the example of strategic classification with adapted best response \cite{li2022state}. In strategic classification, there is a \textit{learner} whose task is to classify a given dataset which is collected from a set of \textit{agents}. Given the knowledge of the classifier, the agents can distort some of their personal features, in order to get classified in a predetermined target class. This scenario arises in various applications, e.g., spam email filtering, and credit score classification. Optimizing the classifier to classify such strategically modified data where the agents modify the data iteratively can be formulated as problem \eqref{eq:mainproblem}. 

Formally, let the classifier be $h(u,\theta)$ where $u\in\mathbb{R}^d$ is the feature and $\theta$ is the parameter to be optimized. $h(u;\cdot):\Theta\to\mathbb{R}$ is potentially nonconvex. Let the loss function be logistic loss which for a sample $x\coloneqq (u,y)$, where $y\in\{-1,1\}$ denotes the corresponding class, is given by,
\begin{align*}
	L(\theta;u,y)=\log\left(1+\exp\left(-h(u,\theta)\right)\right)+(1-y)h(u,\theta)/2.\numberthis\label{eq:logistic}
\end{align*} 	
We use $u_S$, and $u_{-S}$ to denote the parts of feature $u$ which are respectively strategically modifiable, and non-modifiable by the agents. Then the modified feature (the best response) $u_S'$ reported by  the agent is the solution to the following  optimization problem:
\begin{align*}
	u_S'=\argmax_{u_S} \left(h(u;\theta)-c(u_S,u_S')\right), \numberthis\label{eq:featureperturb}
\end{align*}
where $c(u_S,u_S')$ is the cost of modifying $u_S$ to $u_S'$. Let the agents iteratively learn $u_S'$ similar to \cite{li2022state}. Note that unlike \cite{li2022state}, where the authors deploy a logistic regression classifier and the closed form solution of the best response is readily known to the agents,  it may not be the case in general. In that case the agents have to possibly learn the best response $u_S'$ using some iterative optimization algorithm. For example, if the agents use Gradient Ascent then, at every iteration $k$, a set $\cI_k$ of  $n_1\leq M$ randomly chosen agents out of $M$ agents modify their features as:
\begin{align*}
	u_{S,i}^k=\begin{cases}
		u_{S,i}^{k-1}+\alpha\left(\nabla h(u_{S,i}^{k-1};\theta_{k})-\nabla c(u_{S,i}^{k-1},u_{S,i}^0)\right) & i\in \cI_k\\
		u_{S,i}^{k-1}& i\notin \cI_k,
	\end{cases}\numberthis\label{eq:feturemod}
\end{align*} 
where $\alpha>0$ is the stepsize. With a little abuse of notation, we use $\nabla h(u_{S,i}^{k-1};\theta)$  in \eqref{eq:feturemod} to denote  the fact that the gradient is with respect to $u_{S,i}^{k-1}$ while $u_{-S,i}$ remains unchanged. This introduces the state-dependent Markov chain dynamics in the training data. The objective function, analogous to $f(\theta)$ in \eqref{eq:mainproblem}, is $$\underset{\theta\in\Theta}{\min}\expec{L(\theta;x)}{\pi_\theta},$$ where $\pi_\theta$ is the stationary joint distribution of $x$, and $\Theta$ is a convex and compact set, e.g., sparsity inducing constraint $\norm{\theta}_1\leq R$ from some $R>0$. The loss evaluated at a single data point $x$, $L(\theta;x)$, is analogous to $F(\theta;x)$ in \eqref{eq:mainproblem}. \cite{drusvyatskiy2020stochastic}, and \cite{li2022state} study this problem theoretically and empirically respectively in an unconstrained strongly convex setting. Our results takes a step forward towards analyzing this problem in constrained nonconvex setting. We empirically show the performance of the stochastic conditional gradient algorithm on a strategic classification problem in Section~\ref{sec:exp}.

\subsection{Preliminaries and Main Contributions}
Before we present our main contributions, we introduce our convergence criterion. In constrained optimization literature, most commonly used convergence criteria are: \begin{enumerate*}[label=(\roman*)]
    \item \emph{Gradient Mapping} (GM), and \item \emph{Frank-Wolfe Gap} (FW-gap).
\end{enumerate*} 
The \emph{Gradient Mapping} at a point $\bar{\theta} \in \Theta$ is defined, for some $\beta>0$, as 
\begin{equation}\label{eq:definition-gradient-mapping}
    \mathcal{G}_{\Theta}(\bar{\theta}, \nabla f(\bar{\theta}), \beta) \coloneqq
    \beta \left(\bar{\theta} - \Pi_{\Theta}\left(\bar{\theta}- \frac{1}{\beta}\nabla f(\bar{\theta})\right)\right),
\end{equation}
where $\Pi_{\Theta}(x)$ denotes the orthogonal projection of the vector $x$ onto the set $\Theta$, i.e.,
\begin{align*}
    \Pi_{\Theta}\left(\bar{\theta}- \frac{1}{\beta}\nabla f(\bar{\theta})\right)=\underset{y\in\Theta}{\argmin}\left\lbrace \langle \nabla f(\bar\theta),y-\bar{\theta}\rangle+\frac{\beta}{2}\norm{y-\theta}_2^2\right\rbrace.
\end{align*}  
We will use $\Pi_\Theta (\bar{\theta},\nabla f(\bar{\theta}),\beta)$ to denote $\Pi_{\Theta}\left(\bar{\theta}- \nabla f(\bar{\theta})/\beta\right)$ when there is no confusion. 
Note that when $\Theta \equiv \mathbb{R}^d$ we have $\mathcal{G}_{\Theta}(\bar{\theta}, \nabla f(\bar{\theta}), \beta)=\nabla f(\bar{\theta})$. In other words, for constrained optimization gradient mapping plays an analogous role of the gradient for unconstrained optimization. The gradient mapping is a frequently used measure in the literature as a convergence criterion for nonconvex constrained optimization \cite{nesterov2018lectures}. We should emphasize here that although the gradient mapping cannot be computed in the stochastic setting, one can still use it as a convergence measure.

\cite{balasubramanian2022zeroth} shows that the above notion of convergence criterion is closely related to the so-called \emph{Frank-Wolfe Gap}. The FW-gap is defined as
\begin{equation}\label{eq:definition-fw-gap}
    g_{\Theta}(\bar{\theta}, \nabla f(\bar{\theta})) := \underset{y\in\Theta}{\max}~\langle \nabla f(\bar{\theta}) ,  \bar{\theta}-y \rangle.
\end{equation}
The following proposition from \cite{balasubramanian2022zeroth} establishes the relation between the GM and the FW-gap.
\begin{proposition}[\cite{balasubramanian2022zeroth}]
Let $g_{\Theta}(\cdot)$ be the Frank-Wolfe gap defined in \eqref{eq:definition-fw-gap} and $\mathcal{G}_{\Theta}(\cdot)$ be the gradient mapping defined in \eqref{eq:definition-gradient-mapping}. Then, we have
\begin{equation*}
    \| \mathcal{G}_{\Theta}(\bar{\theta}, \nabla f(\bar{\theta}), \beta)\|^2 \leq g_{\Theta}(\bar{\theta}, \nabla f(\bar{\theta})),\qquad \forall \bar{\theta}\in\Theta.
\end{equation*}
Moreover, under standard regularity assumption in smooth optimization (specifically, under Assumptions \ref{aspt:constraint} and  \ref{as:contdiff}), we have
\begin{equation*}
    g_{\Theta}(\bar{\theta}, \nabla f(\bar{\theta})) \leq \frac{L}{\beta} \norm{\mathcal{G}_{\Theta}(\bar{\theta}, \nabla f(\bar{\theta}), \beta)}_2. \numberthis\label{eq:fwgradmapbound}
\end{equation*}
\end{proposition}
In this work we use a suboptimality measure, closely related to both GM and the FW-gap. At point $\bar{\theta}\in\Theta$, we define the suboptimality measure $V(\bar{\theta},z):\mathbb{R}^d\times \mathbb{R}^d\to\mathbb{R}$ as \cite{ghadimi2020single}
\begin{align}
    V(\bar{\theta},z)\coloneqq\norm{\Pi_\Theta\left(\bar{\theta}-z/\beta\right)-\bar{\theta}}_2^2+\norm{z-\nabla f(\bar{\theta})}_2^2,\label{eq:Vdef}
\end{align}
where $z$, formally defined in Algorithm~\ref{alg:asa}, is the moving-average estimate of $\nabla f(\bar{\theta})$.
We show the relation among $V(\theta,z)$, and GM $\mathcal{G}_{\Theta}(\theta, z, \beta)$ in the following proposition.
\begin{proposition}\label{prop:VGMFWrel}
Let $\{z_k\}$ be the sequence generated in Algorithm~\ref{alg:asa}. Then, for $k=1,2,\cdots,N$,
\begin{align*}
   \| \mathcal{G}_{\Theta}(\theta_k, z_k, \beta)\|_2^2\leq \max(2,2\beta^2)V(\theta_k,z_k).
\end{align*}
\end{proposition}
\noindent \begin{proof}[Proof of Proposition~\ref{prop:VGMFWrel}]\label{pf:VGMFWrel}
Using properties of projection onto a convex set, we have 
\begin{align*}
    \| \mathcal{G}_{\Theta}(\theta, z, \beta)\|_2^2&\leq 2\beta^2 \norm{\theta - \Pi_{\Theta}\left(\theta- z/\beta\right)}_2^2+2\beta^2 \norm{\Pi_{\Theta}\left(\theta- z/\beta\right)-\Pi_{\Theta}\left(\theta- \nabla f(\theta)/\beta\right)}_2^2 \\ 
    &\leq \max(2,2\beta^2) V(\theta,z).
\end{align*}
\end{proof}
The main objective of this work is to find an $\epsilon$-stationary solution to~\eqref{eq:mainproblem}, where an $\epsilon$-stationary solution is defined as follows: 
\begin{definition}\label{def:gradmapping}
A point $\bar{\theta}$ is said to be an $\epsilon$-stationary solution to \eqref{eq:mainproblem}, if $\expec{V(\bar{\theta},z)}{}\leq \epsilon$, where the expectation is taken over all the randomness involved in the problem.
\end{definition}
For stochastic Frank-Wolfe-type algorithms, the oracle complexity is measured in terms of number of calls to the Stochastic First-order Oracle (SFO) and the Linear Minimization Oracle (LMO) used to the solve the sub-problems of the algorithm which involves minimizing a linear function over the convex constraint set. Formally, we have the following definition.
\begin{definition}
For a given point $\theta\in\Theta$, SFO returns the stochastic gradient $\nabla F(\theta,x)$. Given a vector $z$, LMO returns a vector $v:=\argmin_{y\in\Theta}\langle z,y\rangle$.
\end{definition}
Hence, in this work, the oracle complexity is measured in terms of the number of calls to SFO and LMO required by the proposed algorithm to obtain an $\epsilon$-stationary solution as in Definition~\ref{def:gradmapping}. With the above preliminaries, we now list our \textbf{main contributions}:
\begin{itemize}[noitemsep]
\item In Theorem~\ref{th:mainthm}, we show that the number of calls to the SFO and LMO required by the stochastic conditional gradient-type method in Algorithm~\ref{alg:asa}, with \textit{state-dependent} Markovian data is of order $\mathcal{O}(\epsilon^{-2.5})$ and $\mathcal{O}(\epsilon^{-5.5})$ respectively in terms of the FW-Gap and the Gradient Mapping criterion. To the best of our knowledge, these are the first oracle complexity results for projection-free one-sample stochastic optimization algorithm for constrained nonconvex optimization in the Markovian setting. Our result also implies an SFO complexity of $\mathcal{O}(\epsilon^{-2.5})$ for projection-based algorithms in terms of the Gradient Mapping criterion. 
\item In Theorem~\ref{th:mainthmhomog}, for the sake of completion, we also show that the number of calls to the SFO and LMO required for the case of \emph{state-independent} Markovian data is of the order $\mathcal{\tilde{O}}(\epsilon^{-2})$ and $\mathcal{\tilde{O}}(\epsilon^{-3})$ respectively. In particular, this turns out to be of the same order as that of $\iid$ data ignoring the logarithmic factors.
\end{itemize}
A summary of the our contributions is provided in Table~\ref{tab:mainresults}. We also empirically evaluate our algorithm on a strategic classification problem with 2-layer neural network classifier and show that the proposed method obtains encouraging results. \textcolor{black}{We provide an experiment on single-index model regression with sparsity-inducing nuclear-norm ball constraint in Section~\ref{sec:tracenorm}.} 
\vspace{-0.1in}
\subsection{Related Work}
\begin{table}[t]
\resizebox{\textwidth}{!}{%
\begin{tabular}{|l|l|ll|llll|}
\hline
                                                     &                        & \multicolumn{2}{l|}{}                                                                                  & \multicolumn{4}{c|}{non-$\iid$}                                                                                                                                                                                       \\ \cline{5-8} 
\multirow{-2}{*}{}                                   & \multirow{-2}{*}{} & \multicolumn{2}{c|}{\multirow{-2}{*}{$\iid$}}                                             & \multicolumn{2}{l|}{State-independent MC}                                                            & \multicolumn{2}{l|}{State-dependent MC}                                                                                            \\\hline
Algorithm                                            & Criterion              & \multicolumn{1}{l|}{SFO}                                     & LMO                                     & \multicolumn{1}{l|}{SFO}                      & \multicolumn{1}{l|}{LMO}                      & \multicolumn{1}{l|}{SFO}                                                               & LMO                                       \\ \hline
1-SFW \cite{zhang2020one}           & FW-gap                 & \multicolumn{1}{l|}{$\mathcal{O}\left(\epsilon^{-3}\right)$} & $\mathcal{O}\left(\epsilon^{-3}\right)$ & \multicolumn{1}{l|}{\xmark}                  & \multicolumn{1}{l|}{\xmark}                  & \multicolumn{1}{l|}{\xmark}                                                           & \xmark                                   \\ \hline
(ASA+ICG) \cite{xiao2022projection} & GM                   & \multicolumn{1}{l|}{$\mathcal{O}\left(\epsilon^{-2}\right)$} & $\mathcal{O}\left(\epsilon^{-3}\right)$ & \multicolumn{1}{l|}{\xmark}                  & \multicolumn{1}{l|}{\xmark}                  & \multicolumn{1}{l|}{\xmark}                                                           & \xmark                                   \\ \hline
\rowcolor[HTML]{F0FFFF} 
(ASA+ICG) [This paper]                                 & GM           &  
\multicolumn{1}{l|}{\xmark}              &            \multicolumn{1}{l|}{\xmark}             & 
\multicolumn{1}{l|}{$\mathcal{\tilde O}\left(\epsilon^{-2}\right)$}              &            \multicolumn{1}{l|}{$\mathcal{\tilde O}\left(\epsilon^{-3}\right)$}             &  \multicolumn{1}{l|}{\cellcolor[HTML]{F0F8FF}$\mathcal{O}\left(\epsilon^{-2.5}\right)$} & $\mathcal{O}\left(\epsilon^{-5.5}\right)$  \\ \hline
\end{tabular}%
}
\caption{Oracle complexity of projection-free one-sample stochastic conditional gradient algorithms for constrained non-convex optimization, to find an $\epsilon$-stationary point.}
\label{tab:mainresults}
\end{table}
\noindent\textbf{Stochastic Optimization with Dependent Data.} Understanding stochastic approximation algorithms like SGD with dependent data in the asymptotic setting has been well-explored in the optimization literature. We refer to~\cite{Kushner2003book,borkar2009stochastic,  benveniste2012adaptive} for a text-book introduction to such classical results. A few recent results include~\cite{andrieu2005stability, tadic2017asymptotic}. In the unconstrained non-asymptotic setting, \cite{duchi2012ergodic} studies convex optimization with ergodic data sequence. \cite{dorfman2022adapting} uses multi-level gradient estimator and analyzes AdaGrad for nonconvex optimization with Markovian Data. Block coordinate descent with homogeneous Markov chain has been analyzed in \cite{sun2020markov} for nonconvex unconstrained optimization. \cite{drusvyatskiy2020stochastic} studies stochastic optimization with state-dependent Markov data for strongly convex functions in the context of strategic classification.  

Sample-average approximation algorithms for constrained convex optimization with $\phi$-mixing data was considered in \cite{wang2021sample}. \cite{sun2018markov}, and \cite{alacaoglu2022convergence} analyze projected SGD for constrained nonconvex optimization with time-homogeneous Markov chain. We emphasize that the above works, except for \cite{drusvyatskiy2020stochastic} do not  consider state-dependent Markov data. Furthermore, unlike \cite{drusvyatskiy2020stochastic}, we study constrained nonconvex optimization problems. 

There also exists work in the reinforcement learning literature on understanding stochastic optimization with Markovian data; see, for example,~\cite{bhandari2018finite, srikant2019finite, xu2019two, kaledin2020finite,doan2020convergence,xiong2021non,durmus2021stability}. However, such works are invariably focused on specific objective functions arising in the reinforcement learning setup, while our focus is on obtaining results for a general class of functions.\\

\noindent\textbf{Conditional Gradient-Type Method.} There has been significant recent advancements in understanding conditional gradient algorithm in the machine learning literature; see, for example,~\cite{migdalas1994regularization,jaggi2013revisiting,lacoste2015global, lacoste2015global, harchaoui2015conditional, garber2021improved, beck2017linearly}, for a non-exhaustive list.  ~\cite{hazan2012projection, hazan2016variance} provided expected oracle complexity results for stochastic conditional gradient algorithm in the stochastic convex setup. Better rates were provided by a sliding procedure in~\cite{lan2016conditional}. In the non-convex setting, ~\cite{reddi2016stochastic, yurtsever2019conditional, hazan2016variance} considered variance reduced stochastic conditional gradient algorithms, and provided expected oracle complexities. \cite{qu2018non} analyzed the sliding algorithm in the non-convex setting and provided results for the gradient mapping criterion. All of the above works use increasing orders of mini-batch based gradient-estimate.

To avoid mini-batches, a moving-average gradient estimator based on only one-sample in each iteration for a stochastic conditional gradient-type algorithm was proposed in \cite{mokhtari2020stochastic} and~\cite{zhang2020one} for the convex and non-convex setting. However, several restrictive assumptions have been made in ~\cite{mokhtari2020stochastic} and~\cite{zhang2020one}. Specifically,~\cite{zhang2020one} requires that the stochastic gradient $G_1(x,\xi_1)$ has uniformly bounded function value, gradient-norm, and Hessian spectral-norm, and the distribution of the random vector $\xi_1$ has an absolutely continuous density $p$ such that the norm of the gradient of $\log p$ and spectral norm of the Hessian of $\log p$ has finite fourth and second-moments respectively. A recent work~\cite{xiao2022projection} provided similar convergence results under significantly weaker conditions. 
\section{Main Results}\label{sec:assumption}
\subsection{Methodology}
We use the moving-average based single-time scale algorithm as proposed by~\cite{ghadimi2020single} for constrained optimization.  The overall procedure is provided in Algorithm~\ref{alg:asa}, and \ref{alg:icg}. 
\begin{algorithm}[h]
\caption{Inexact Averaged Stochastic Approximation (I-ASA)}\label{alg:asa}
\textbf{Input:} $z_0,\theta_0\in \mathbb{R}^d$, $\eta_k=(N+k)^{-a}$, $1/2<a<1$, $\beta$. 
\begin{algorithmic}[h]
\State \textbf{for} $k=1,2,\cdots,N$ \textbf{do}
\State $y_k=\begin{cases}
  \underset{y\in\Theta}{\min}\left\lbrace \langle z_k,y-\theta_k\rangle+\frac{\beta}{2}\norm{y-\theta_k}_2^2\right\rbrace & \text{(Projection)} \\
  \text{ICG}(z_k,\theta_k,\beta,t_k,\omega) & \text{(No Projection)}
\end{cases}$ 
    \State $\theta_{k+1}=\theta_k+\eta_{k+1}(y_k-\theta_k)$
    \State $z_{k+1}=(1-\eta_{k+1})z_{k}+\eta_{k+1}\nabla F(\theta_k,x_{k+1})$
    \State \textbf{end for}
\end{algorithmic}
\textbf{Output:} $\theta_R$ where $P(R=i)=\frac{\eta_i}{\sum_{j=1}^N\eta_j}$ for $i=1,2,\cdots,N$.
\end{algorithm}
\begin{algorithm}[h]
\caption{Inexact Conditional Gradient (ICG)}\label{alg:icg}
\textbf{Input:} $z, \theta,\beta,t,\omega$.
\begin{algorithmic}[h]
\State \textbf{Set} $w_0=\theta$
\State \textbf{for} $i=1,2,\cdots,{t-1}
$ \textbf{do}
\State \text{Find} $v_i$ \text{such that}\\
\vspace{-0.15in}
\begin{align*}
    \left\langle v_i,z+\beta(w_i-\theta)\right\rangle\leq \argmin_{v\in\Theta} \left\langle v,z+\beta(w_i-\theta)\right\rangle+\beta\omega\mathcal{D}_\Theta^2/(i+2)
\end{align*}
    \State $w_{i+1}=(1-\mu_i)w_i+\mu_iv_i$ \text{where} $\mu_i=\frac{2}{i+2}$
    \State \textbf{end for}
\end{algorithmic}
\textbf{Output:} $w_t$
\end{algorithm}
 It is worth emphasizing that the above approach is similar to \texttt{ASA} algorithm introduced in \cite{ghadimi2020single} except that we do not assume the knowledge of the exact minimizer, which is the projection of $\theta_k-z_k/\beta$ on to $\Theta$, of the following subproblem:
\begin{align*}
    \underset{y\in\Theta}{\min}\left\lbrace \langle z_k,y-\theta_k\rangle+\frac{\beta}{2}\norm{y-\theta_k}_2^2\right\rbrace.\numberthis\label{eq:mainalgsubprob}
\end{align*}
Instead, at iteration $k$, Algorithm~\ref{alg:icg} finds an approximate solution to \eqref{eq:mainalgsubprob} based on the conditional gradient algorithm. The main idea is to replace costly projection operator by the Inexact Conditional Gradient (ICG) method which solves \eqref{eq:mainalgsubprob} approximately with access to LMO which is often much cheaper and simpler to compute. Such an approach was also used recently in \cite{xiao2022projection} in the context of $\iid$ data. 
\subsection{Assumptions}
We now introduce the assumption that we make on the optimization problem~\eqref{eq:mainproblem}. We refer to~\cite{douc2018markov} for a textbook introduction to additional details regarding several assumptions below. Let $\cF_k$ be the filtration generated by $\{\theta_0,\cdots,\theta_k,z_0,\cdots,z_k,x_1,\cdots,x_k\}$. For any mapping $g:\mathbb{R}^d\to\mathbb{R}^d$ define the norm with respect to a function $\mathcal{V}:\mathbb{R}^d\to [1,\infty)$ as
\begin{align*}
    \norm{g}_\mathcal{V}=\underset{x\in\mathbb{R}^d}{\sup}\frac{\norm{g(x)}_2}{\mathcal{V}(x)},
\end{align*}
and let $L_\mathcal{\mathcal{V}}=\{g:\mathbb{R}^d\to\mathbb{R}^d,\sup_{x\in\mathbb{R}^d}\norm{g}_\mathcal{V}<\infty\}$.

\begin{assumption}[Constraint set]\label{aspt:constraint}
The set $\Theta\subset \mathbb{R}^d$ is convex and closed with $\underset{x,y \in \Theta}{\max}~\|x-y\|_2 \leq D_\Theta$, for some $D_\Theta>0$.
\end{assumption}
\begin{assumption}\label{as:contdiff}
Let $f$ be a continuously differentiable function. 
\end{assumption}
\begin{assumption}\label{as:finitenoisevar}
Let $\xi_{k+1}(\theta_k,x_{k+1})\coloneqq\nabla F(\theta_k,x_{k+1})-\nabla f(\theta_k)$. Then,
\begin{align*}
    \expec{\norm{\xi_{k+1}(\theta_k,x_{k+1})}_2^2|\cF_{k}}{}\leq \sigma_1^2, \qquad \expec{\norm{\nabla F(\theta_k,x_{k+1})}_2^2|\cF_{k}}{}\leq \sigma_2^2.
\end{align*}
Let $\sigma^2\coloneqq \max(\sigma_1^2,\sigma_2^2)$.
\end{assumption}
\begin{assumption}\label{as:noise}
Let $\{x_k\}_k$ be a Markov chain controlled by $\theta$, i.e., there exists a transition probability kernel $P_\theta(\cdot,\cdot)$ such that
\begin{align*}
    \mathbb{P}(x_{k+1}\in B|\theta_0,x_0,\cdots,\theta_k,x_k)=P_{\theta_k}(x_k,B),
\end{align*}
almost surely for any Borel-measurable set $B\subseteq\mathbb{R}^d$ for $k\geq 0$. For any $\theta\in\Theta$, $P_\theta$ is irreducible and aperiodic. Additionally, there exists a function $\mathcal{V}:\mathbb{R}^d\to [1,\infty)$ and a constant $\alpha_0\geq 2$ such that for any compact set $\Theta'\subset\Theta$, we have the following.
\begin{enumerate}[label=(\alph*)]
    \item \label{eq:asa31} There exist a set $C\subset \mathbb{R}^d$, an integer $l$, constants $0<\lambda<1$, $b$, $\kappa$, $\delta>0$, and a probability measure $\nu$ such that,
    \begin{align}
        \sup_{\theta\in\Theta'}P_\theta^l\mathcal{V}^{\alpha_0}(x)&\leq \lambda \mathcal{V}^{\alpha_0}(x)+bI(x\in C)\quad \forall x\in\mathbb{R}^d,\\
        \sup_{\theta\in\Theta'}P_\theta \mathcal{V}^{\alpha_0}(x)&\leq\kappa \mathcal{V}^{\alpha_0}(x)\quad \forall x\in\mathbb{R}^d,\\
        \inf_{\theta\in\Theta'}P_\theta^l(x,A)&\geq \delta\nu(A) \forall x\in C, \forall A\in \mathcal{B}_{\mathbb{R}^d}.
    \end{align}
    where $\mathcal{B}_{\mathbb{R}^d}$ is the Borel $\sigma$-algebra over $\mathbb{R}^d$, and for a function $m$, $P_\theta m(x) \coloneqq \int P_\theta (x,y) m(y)\, dy $.
    \item There exists a constant $c>0$, such that, for all $x\in\mathbb{R}^d$,
    \begin{align}
        \sup_{\theta\in\Theta'}\norm{\nabla F(\theta,x)}_\mathcal{V}&\leq c,\\
        \norm{\nabla F(\theta,x)-\nabla F(\theta',x)}_\mathcal{V}&\leq c\norm{\theta-\theta'}_2\qquad \forall (\theta,\theta')\in\Theta'.
    \end{align}
    \item There exists a constant $c>0$, such that, for all $(\theta,\theta')\in\Theta'\times\Theta'$,
    \begin{align}
        \norm{P_\theta g-P_{\theta'} g}_\mathcal{V}&\leq c\norm{g}_\mathcal{V}\norm{\theta-\theta'}_2\quad \forall g\in L_\mathcal{V}\\
         \norm{P_\theta g-P_{\theta'} g}_{\mathcal{V}^{\alpha_0}}&\leq c\norm{g}_{\mathcal{V}^{\alpha_0}}\norm{\theta-\theta'}_2\quad \forall g\in L_{\mathcal{V}^{\alpha_0}}. \label{eq:transitionsmooth}
    \end{align}
\end{enumerate}
\end{assumption}
Some comments regarding the assumptions are in order. Assumption~\ref{aspt:constraint}, and Assumption~\ref{as:contdiff} are common in the literature on smooth constrained optimization \cite{ghadimi2020single,xiao2022projection,alacaoglu2022convergence,zhang2020one}. Assumption~\ref{aspt:constraint}, and Assumption~\ref{as:contdiff} together imply the Lipschitz continuity of $f(\cdot)$, i.e., there is a constant $L>0$ such that for any $\theta_1,\theta_2\in\Theta$, we have $    \abs{f(\theta_1)-f(\theta_2)}\leq L\norm{\theta_1-\theta_2}_2$.
Assumption~\ref{as:finitenoisevar} is common in stochastic optimization literature. Assumption~\ref{as:noise}(a) is a frequently used assumption in Markov chain literature. It implies that for every $\theta\in\Theta$, there exists a stationary distribution $\pi_{\theta}(x)$, and the chain is $\mathcal{V}^{\alpha_0}$-uniformly ergodic \cite{andrieu2005stability}. Assumption~\ref{as:noise}(c) provides smoothness guarantee on the function $f(\cdot)$. More formally, we have the following proposition.
\begin{proposition}[\L continuous gradient \cite{andrieu2005stability}]
Let Assumption~\ref{as:noise} be true. Then $f(\cdot)$ has Lipschitz continuous gradient, i.e., there is a constant $L_G>0$ such that for any $\theta_1,\theta_2\in\Theta$:
\begin{align*}
    \norm{\nabla f(\theta_1)-\nabla f(\theta_2)}_2\leq L_G\norm{\theta_1-\theta_2}_2.\numberthis\label{eq:lipgrad}
\end{align*}
\end{proposition}
Finally, the most important implication of Assumption~\ref{as:noise} is that it ensures the existence and regularity of a solution $u(\theta,x)$ to the Poisson equation,
\begin{align}\label{eq:poissoneq}
u(\theta,x)-P_{\theta}u(\theta,x)=\nabla F(\theta,x)-\nabla f(\theta), 
\end{align}
associated to the transition kernel $P_\theta$, where $P_\theta u(\theta,x)\coloneqq \int_{\mathbb{R}^d}u(\theta,x')P_\theta (x,x')dx'$. Solutions of Poisson equation have been crucial in analyzing additive functionals of Markov chain (see \cite{andrieu2005stability} for details). In this work, the Poisson equation solution facilitates a decomposition of the noise as presented in Lemma~\ref{lm:noisedecompbound} which is a key component of our analysis.

\subsection{State-dependent Markov Chain}
We now present our main result on the oracle complexity to establish a bound on $\expec{V(\theta_k,z_k)}{}$.  We emphasize that our results are not limited to \texttt{ICG} method but are valid for any method which can solve \eqref{eq:mainalgsubprob} within an error of the order of $\{\eta_k\}$.
\textcolor{black}{\begin{theorem}\label{th:mainthm}
Let Assumption~\ref{aspt:constraint}-\ref{as:noise} be true. Then, for Algorithm~\ref{alg:asa},
\begin{enumerate}[label=(\alph*)]
    \item when a projection operator is available,
    choosing 
    \begin{align}
    \eta_k=(N+k)^{-3/5}, \quad \beta=1 \label{eq:parchoice}
\end{align}
for $k=1,2,\cdots,N$ we have  
\begin{align*}
    \expec{V(\theta_R,z_R)}{}=\mathcal{O}\left(N^{-\frac{2}{5}}\right),
\end{align*}
\item when Algorithm~\ref{alg:icg} is used to solve \eqref{eq:mainalgsubprob},choosing 
    \begin{align}
    \eta_k=(N+k)^{-3/5}, \quad t_k=\eta_k^{-2}, \quad \beta=1, \quad \omega=1, \quad \mu_i=2/(i+2) \label{eq:parchoice}
\end{align}
for $k=1,2,\cdots,N$ we have  
\begin{align*}
    \expec{V(\theta_R,z_R)}{}=\mathcal{O}\left(N^{-\frac{2}{5}}\right),
\end{align*}
\end{enumerate}
where the expectations are taken with respect to all the randomness of the algorithm, and an independent integer random variable $R\in\{1,2,\cdots,N\}$ with probability mass function,
\begin{align*}
    P\left(R=k\right)=\eta_k/\sum_{k=1}^N\eta_k \quad k\in \{1,2,\cdots,N\}.
\end{align*}
\end{theorem}}
\begin{remark}
Note that total number of LMO calls are $\sum_{k=1}^Nt_k=\sum_{k=1}^Nt_k=\sum_{k=1}^N(N+k)^{2a}=\mathcal{O}(N^{11/5})$. In other words, to achieve $\norm{\mathcal{G}_\Theta(\theta_R,\nabla f(\theta_R),\beta)}_2^2\leq \expec{V(\theta_R,z_R)}{}\leq \epsilon$, SFO and LMO complexities are respectively $\epsilon^{-2.5}$, and $\epsilon^{-5.5}$. Note that the SFO complexity will be $\epsilon^{-2.5}$ as long as one has an approximation of the projection operator with approximation error $\mathcal{O}(\eta_k)$. 
\end{remark}
\vspace{-0.1in}
\begin{remark}
\textcolor{black}{In Theorem~\ref{th:mainthm}, one obtains sublinear rate $\max(N^{a-1},N^{2-4a})$ with $\eta_k=(N+k)^{-a}$ for $1/2<a<1$. Choosing $a=3/5$ provides the fastest rate of convergence.}
\end{remark}
Before sketching the outline of the proof, we present the following lemma which provides a decomposition of the noise $\xi_{k}(\theta_{k-1},x_k)$ -- one of the key results used in the proof of the main theorem. The lemma and its proof are almost same as Lemma A.5 in \cite{liang2010trajectory} with the only difference that unlike \cite{liang2010trajectory}, where the iterates are of SGD, we need to prove it for the iterates of Algorithm~\ref{alg:asa}. We provide the proof in Appendix~\ref{pf:noisedecompbound}.   
\begin{lemma}\label{lm:noisedecompbound}
Let Assumption~\ref{aspt:constraint}-\ref{as:noise} be true. Then the following decomposition takes place:
\begin{align*}
    \xi_{k}(\theta_{k-1},x_{k})=e_{k}+\nu_{k}+\zeta_{k},
\end{align*}
where, $\{e_k\}_k$ is martingale difference sequence, $\expec{\norm{\nu_k}_2}{}\leq \eta_{k}$, and $\zeta_k=(\tilde{\zeta}_k-\tilde{\zeta}_{k+1})/\eta_k$, 
where $\{\tilde{\zeta}_k\}_k$ is a sequence such that $\expec{\|\tilde{\zeta}_k\|_2}{}\leq \eta_k$. The exact expressions of $\{e_k\}_k$, $\{\nu_k\}_k$, and $\{\tilde{\zeta}_k\}_k$ are provided in~\eqref{eq:compdef}. 
\end{lemma}

\noindent \textbf{Outline of the proof of Theorem~\ref{th:mainthm}:} The key step in the analysis of Algorithm~\ref{alg:asa} involves controlling the expectation of interaction with noise of the form $\langle \nabla f(\theta_{k})-\nabla f(\theta_{k-1}),\xi_{k+1}(\theta_k,x_{k+1})\rangle$. For $\iid$ or martingale difference data it is easy to control because $$\expec{\langle \nabla f(\theta_{k})-\nabla f(\theta_{k-1}),\xi_{k+1}(\theta_k,x_{k+1})\rangle|\cF_k}{}=0.$$ However, this is no longer true for Markov chain data. To resolve the issue, first notice that under our assumptions, the noise sequence $\xi_k$ can be decomposed into the sum of a martingale difference sequence $\{e_k\}$ and some residual terms $\{\nu_k\}$, and $\{\zeta_k\}$ as shown in Lemma~\ref{lm:noisedecompbound}. Then the key step is to introduce a different sequence of hypothetical iterates $(\tt_k,\ty_k,\tz_k)$ for which the noise is small enough so that we can bound $\expec{V(\tt_k,\tz_k)}{}$, and then show that these hypothetical iterates and the original sequence generated by Algorithm~\ref{alg:asa} are close enough so that $\expec{V(\theta_k,z_k)}{}$ is of the same order as $\expec{V(\tt_k,\tz_k)}{}$. This forms the main novelty in our analysis.

Specifically, the hypothetical sequence that we consider is given by:
\begin{align}
& \tt_0=\theta_0 \quad \tz_0=z_0\\
    &\ty_k=\argmin_{y\in\Theta}\left\lbrace\left\langle \tz_k,y-\tt_k\right\rangle+\frac{\beta}{2}\norm{y-\tt_k}_2^2\right\rbrace \label{eq:tilde1}\\
    &\tt_{k+1}=\tt_k+\eta_{k+1}(\ty_k-\tt_k)\label{eq:tilde2}\\
    &\tz_{k+1}=z_{k+1}+\tilde{\zeta}_{k+2}.\label{eq:tilde3}
\end{align}
This also means that we have
\begin{align*}
    \tz_{k+1}=(1-a\eta_{k+1})\tz_k+a\eta_{k+1}\left(\nabla f(\theta_k)+\tilde{\epsilon}_{k+1}\right),\numberthis\label{eq:ztildek1update}
\end{align*}
where,
$\tilde{\epsilon}_{k}=e_{k}+\nu_k+\tilde{\zeta}_k$. Note that by Lemma~\ref{lm:noisedecompbound}, $\expec{e_k}{}=0$, and $\expec{\norm{\nu_k+\tilde{\zeta}_{k}}_2}{}\leq \eta_k$. First we show that by choosing $$\eta_k=(N+k)^{-a}, \quad 1/2<a<1 \quad \text{and} \quad t_k=1/\eta_k^2,$$ one has that 
$$
\expec{\|\tt_k-\theta_k\|_2^2}{}=\mathcal{O}\left(N^{2-4a}\right),\quad\text{and}\quad \expec{V(\theta_k,z_k)}{}\leq 2\expec{V(\tt_k,\tz_k)}{}+\mathcal{O}\left(N^{2-4a}\right).$$ 
Then we establish a similar bound on $V(\tt_k,\tz_k)$. Combining the above two facts proves Theorem~\ref{th:mainthm}. The full proof is deferred to Appendix~\ref{pf:mainthm}.
\subsection{State-independent Markov Chain}
While our main goal in this work is to analyze Algorithm~\ref{alg:asa} for constrained nonconvex optimization with state-dependent Markov chain data, we provide the following result on the complexity of Algorithm~\ref{alg:asa} for Markov chain data with state-independent transition kernel for the sake of completion. Here we use $P$ to denote the transition kernel (as opposed to $P_\theta$ for state-dependent kernel). Note that under Assumption~\ref{as:noise}(a), for each $\theta$, the chain is $\mathcal{V}$-uniformly ergodic, and hence, exponentially mixing \cite{meyn2012markov} in the following sense.
\begin{definition}
A Markov chain is said to be exponentially mixing, if there exists $C,r>0$ such that, for any initial state $x$, 
\begin{align*}
    \norm{P^n(x,\cdot)-\pi}_\mathcal{V}\leq C\exp(-rn),\numberthis\label{eq:mixing}
\end{align*}
where $P^n(x,\cdot)$ is the distribution of $X_n$ with initial state $X_0=x$. 
\end{definition}
Now we present our result on the complexity of Algorithm~\ref{alg:asa} to find an $\epsilon$-stationary solution to \eqref{eq:mainproblem} for exponentially-mixing Markov chain data with state-independent transition kernel.
\textcolor{black}{
\begin{theorem}\label{th:mainthmhomog}
Let Assumption~\ref{aspt:constraint}-\ref{as:finitenoisevar} be true. Let Assumption~\ref{as:noise}(a)-(b) be true with $P_\theta$ replaced by $P$. Then, for Algorithm~\ref{alg:asa},
\begin{enumerate}[label=(\alph*)]
    \item when the projection operator is available, choosing
\begin{align}
    \eta_k=1/\sqrt{N}, \quad \beta=1 \label{eq:parchoicehomo}
\end{align}
for $k=1,2,\cdots,N$ we have  
\begin{align*}
    \expec{V(\theta_R,z_R)}{}=\mathcal{O}\left(\log N/\sqrt{N}\right),
\end{align*}
\item when Algorithm~\ref{alg:icg} is used,
choosing
\begin{align}
    \eta_k=1/\sqrt{N}, \quad t_k=\ceil{\sqrt{k}}, \quad \beta=1, \quad \omega=1, \quad \mu_i=2/(i+2) \label{eq:parchoicehomo}
\end{align}
for $k=1,2,\cdots,N$ we have  
\begin{align*}
    \expec{V(\theta_R,z_R)}{}=\mathcal{O}\left(\log N/\sqrt{N}\right),
\end{align*}
\end{enumerate}
where the expectation is taken with respect to all the randomness of the algorithm, and an independent integer random variable $R\in\{1,2,\cdots,N\}$ whose probability mass function is given by,
\begin{align*}
    P\left(R=k\right)=\eta_k/\sum_{k=1}^N\eta_k \quad k\in \{1,2,\cdots,N\}.
\end{align*}
\end{theorem}}
We defer the proof to the Appendix. 
\begin{remark}
 To find an $\epsilon$-stationary point, the total number of calls to SFO and LMO are $\mathcal{\tilde O}\left(\epsilon^{-2}\right)$, and $\mathcal{\tilde O}\left(\epsilon^{-3}\right)$, where $\mathcal{\tilde O}(\cdot)$ denotes the order ignoring logarithmic factors.  
\end{remark}
\vspace{-0.05in}
\begin{remark}
The authors of \cite{alacaoglu2022convergence} obtain the same rate as in Theorem~\ref{th:mainthmhomog} for constrained (but projection-based) nonconvex optimization with state-independent exponentially mixing data. In the state-dependent case, since the transition kernel of the Markov chain is controlled by $\theta_k$, and the transition kernel is assumed to be only Lipschitz smooth in $\theta$ \eqref{eq:transitionsmooth}, the chain does not necessarily exponentially mix. In the state-independent case, since the chain mixes exponentially we obtain the same rate as well. While their results are for projection-based algorithms, we analyze a projection-free LMO-based algorithm since LMO is often computationally cheaper than projection.
\end{remark} 
\vspace{-0.1in}
\section{Experimental Evaluation}\label{sec:expev}
\subsection{Strategic Classification}\label{sec:exp}
\begin{figure}[t]\label{fig:scnn}
    \centering
    \includegraphics[width=0.7\textwidth]{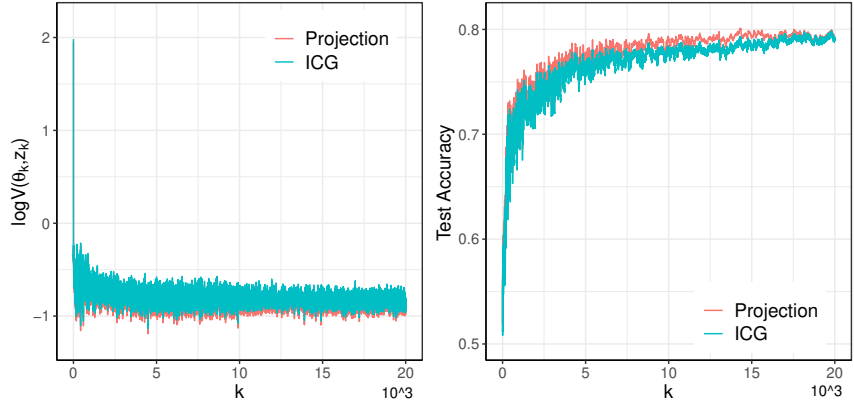}
    \caption{Strategic Classification: (\textit{Left}): Performance of Algorithm~\ref{alg:asa} with and without the projection operator. (\textit{Right}): Test Accuracy with Algorithm~\ref{alg:asa} with and without the projection operator.}
    \label{fig:my_label}
\end{figure}
In this section we illustrate our algorithm on the strategic classification problem as described in Section~\ref{sec:motivate} with the \texttt{GiveMeSomeCredit}\footnote{Available at \url{https://www.kaggle.com/c/GiveMeSomeCredit/data}} dataset. The main task is a credit score classification problem where the bank (learner) has to decide whether a loan should be granted to a client.  Given the knowledge of the classifier the clients (agents) can distort some of their personal traits  in order to get approved for a loan. Here we use a  2-layer neural network with width $m$ as the classifier, given by 
\begin{align*}
	h(u;\mathcal{W},\mathcal{A},\mathcal{B})=\sum_{i=1}^m\cA_i\upsilon(\cW_i^\top u+\cB_i),
\end{align*} 
where $\upsilon(\cdot)$ is the activation function, $\cW_i\in \mathbb{R}^d$ and  $$\cW=[\cW_1,\cW_2,\cdots,\cW_m]^\top\in\mathbb{R}^{m\times d}, \cA=\begin{pmatrix}
	\cA_1, \cA_2,\cdots, \cA_m
\end{pmatrix}\in\mathbb{R}^{m}, \cB=\begin{pmatrix}
\cB_1, \cB_2,\cdots, \cB_m
\end{pmatrix}\in\mathbb{R}^{m}.$$ We will use $\theta$ to collectively denote $(\cW,\cA,\cB)$. We impose the constraint of sparsity on the classifier given by $\|\theta\|_1\leq R$ for some $R>0$. As loss function we consider logistic loss as shown in \eqref{eq:logistic}. We consider a quadratic cost given by $c(u,u')=\norm{u_S-u_S'}_2^2/(2\lambda)$ where $\lambda$ is the sensitivity of the underlying distribution on $\theta$. We assume that the agents iteratively learn $u_S'$ similar to \cite{li2022state}. Note that unlike \cite{li2022state}, the closed form of best response is not known here. So we assume that the agents use Gradient Ascent (GA) to learn the best response. For $\|\theta\|_1\leq R$ constraint, the LMO in Algorithm~\ref{alg:icg} at iteration $k$ is given by,
\begin{align*}
    i=\underset{j=1,\cdots,d}{\argmax}\abs{q_j}\quad \text{LMO}=-R \sign\left(q_i\right),
\end{align*}
where $q=z+\beta(w_k-\theta)$, and $q_j$ is the $j$-th coordinate of $q$.
We select a subset of randomly chosen $M=2000$ samples (agents) such that the dataset is balanced. Each agent has $10$ features. Note that since Algorithm~\ref{alg:asa} computes the gradient on one sample at every iterate, the computation time is independent of the total number of agents. We assume that the agents can modify Revolving Utilization, Number of Open Credit Lines, and Number of Real Estate Loans or Lines. In this experiment we set $n_1=200$. Similar to \cite{li2022state}, we set $\alpha=0.5\lambda$, and $\lambda=0.01$.  For the classifier, the activation function is chosen as \emph{sigmoidal}, and $m=400$. We set $N=20000$, and $R=4000$. All the parameters of Algorithm~\ref{alg:asa} are chosen as described in \eqref{eq:parchoice}. Figure~\ref{fig:scnn} shows that Algorithm~\ref{alg:asa} finds an $\epsilon$-stationary point of the strategic classification problem. We show that Algorithm~\ref{alg:asa} performs comparably with Averaged Stochastic Approximation with the projection operator. Each curve in Figure~\ref{fig:scnn} is an average of 50 repetitions. 
\subsection{Single Index Model with Trace-norm Ball Constraint}\label{sec:tracenorm}
In this section we illustrate our algorithm on a synthetic example on single-index model regression with a nuclear-norm constraint on the model parameter. Let $\norm{\cdot}_*$ denote the nuclear norm. The features $\{x_k\}_k\in\mathbb{R}^{d_1\times d_2}$ are a matrix-valued time-series given by,
\begin{align*}
    x_{k}=Ax_{k-1}+E_k+W_k\upsilon \theta_k,
\end{align*}
where $A\in\mathbb{R}^{d_1\times d_1}$ matrix with spectral radius less than $1$, $E_k\in\mathbb{R}^{d_1\times d_2}$ is the noise matrix with each entry of $E_k$ is $\iid$ $N(0,1)$ random variable, $W_k$ is a $Bernoulli(0.5)$ random variable, and $\upsilon\in\mathbb{R}$. For a fixed $\theta_k=\theta$, $\{x_k\}_k$ has a stationary distribution as shown in Proposition 1 of \cite{chen2021autoregressive}. $\{E_k\}_k$, and $\{W_k\}_k$ are $\iid$ sequence. This Markov chain follows conditions
(b) and (c) of Assumption~\ref{as:noise} since the evolution of $x_k$ only involves linear terms in $\theta_k$. The responses $\{y_k\}_k$ are generated according to the following single index model,
\begin{align*}
    y_k=g(x_k^\top\theta^*)+\tilde{E}_k,
\end{align*}
where $\{\tilde{E}_k\}_k$ is an $\iid$ sequence of standard normal random variables, $\theta^*\in\mathbb{R}^{d_1\times d_2}$ is a matrix with $\norm{\theta^*}_*\leq 1$, and $g(\cdot):\mathbb{R}\to\mathbb{R}$ is the link function. For this experiment we choose $g(x)=3x+5\sin(x)$. Since $y_k$ only depends on $x_k$, and $g$ is a Lipschitz continuous function of $\theta$, Assumption~\ref{as:noise} holds for $(x_k,y_k)$. It is easy to see that Assumptions ~\ref{aspt:constraint} - \ref{as:finitenoisevar} holds for this example. The constraint set is given by $\norm{\theta}_*\leq 1$, i.e., we assume that $\theta^*$ has a low-rank structure. The goal is to minimize the expected squared loss with the constraint $\norm{\theta}_*\leq 1$, i.e., 
\begin{align*}
    \underset{\norm{\theta}_*\leq 1}{\min}~\expec{(y-g(x^\top\theta))^2}{}.\numberthis\label{eq:trnormprob}
\end{align*}
The advantages of conditional-gradient based method for nuclear-norm ball constrained problems have been studied extensively \cite{jaggi2010simple,jaggi2013revisiting,harchaoui2015conditional}. The main advantage of ICG-based method is that calculating the LMO in this case requires computation of the leading singular vector of gradient matrix whereas to calculate the projection on the trace-norm ball one needs to compute the complete singular value decomposition. Let $u_1$, $v_1$ are the leading left and right singular vectors of the noisy gradient matrix evaluated at $(\theta;x,y)$, $-2(y-g(x^\top\theta))g'(x^\top\theta)x$. Then the LMO is given by, 
\begin{align*}
    \text{LMO}=-u_1v_1^\top.
\end{align*}
For this experiment we choose $d_1=10$, $d_2=20$, $\upsilon=0.1$, and $N=2000$. Rest of the parameters of Algorithm~\ref{alg:asa} are chosen according to Theorem~\ref{th:mainthm}. In Figure~\ref{fig:tracenorm}, we compare the projection-based and ICG based version of Algorithm~\ref{alg:asa} with respect to $V(\theta_k,z_k)$, test Mean Squared Error (MSE), and $\norm{\theta_k-\theta^*}_F$ where $\norm{\cdot}_F$ is the Fröbenius norm. Figure~\ref{fig:tracenorm} shows that the performance of projection-based and the ICG-based versions of Algorithm~\ref{alg:asa} are almost same. Each plot in Figure~\ref{fig:tracenorm} is the average of 50 repetitions.
\begin{figure}[t]\label{fig:tracenorm}
    \centering
    \includegraphics[width=\textwidth]{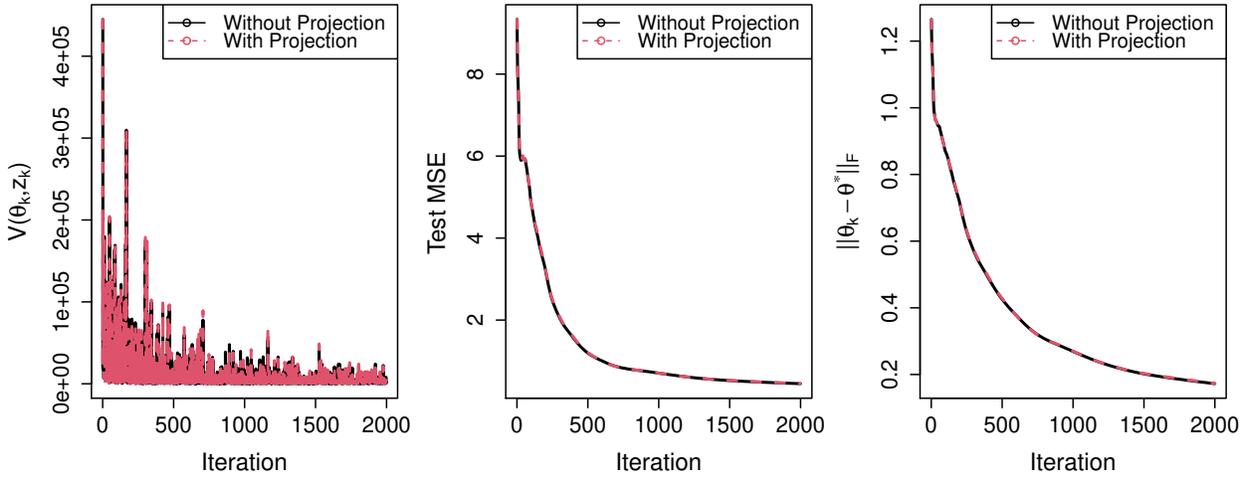}
    \caption{Single-index model with nuclear-norm constraint: (\textit{Left}): Performance of Algorithm~\ref{alg:asa} with and without the projection operator. (\textit{Middle}): Test Mean Squared Error (MSE) with Algorithm~\ref{alg:asa} with and without the projection operator. (\textit{Right}): $\norm{\theta-\theta^*}_F$ with Algorithm~\ref{alg:asa} with and without the projection operator.}
    \label{fig:tracenorm}
\end{figure}
\vspace{-0.1in}
\section{Discussion}
In this work we provide oracle complexity results for the stochastic conditional gradient algorithm to find an $\epsilon$-stationary point of a constrained nonconvex optimization problem with state-dependent Markovian data. In Theorem~\ref{th:mainthm}, we show that the number of calls to the SFO and LMO required by the stochastic conditional gradient-type method in Algorithm~\ref{alg:asa}, with \textit{state-dependent} Markovian data, is $\mathcal{O}(\epsilon^{-2.5})$ and $\mathcal{O}(\epsilon^{-5.5})$ respectively. To the best of our knowledge, these are the first oracle complexity results in this setting. In Theorem~\ref{th:mainthmhomog}, we  show also that the SFO and LMO complexity in the case of state-independent Markovian data is $\mathcal{\tilde{O}}(\epsilon^{-2})$ and $\mathcal{\tilde{O}}(\epsilon^{-3})$ respectively, which matches the corresponding results in the $\iid$ setting. 

There are various avenues for further extensions. Improving the established complexities and/or proving matching lower bounds on the oracle complexity of projection-free algorithms in the Markovian setting is extremely interesting. It is also intriguing to establish upper and lower bounds on the oracle complexity for more general types of dependent data sequences arising in applications, including $\phi$ and $\alpha$-mixing sequences. Yet another exciting direction is that of designing algorithms adaptive to the dependency in the data that achieve potentially better oracle complexity bounds.

\appendix
\appendix
\section{Proofs for the State-dependent Case}

Before proving Lemma~\ref{lm:noisedecompbound}, we present the following result on Poisson equation solution from \cite{liang2010trajectory}, and \cite{andrieu2005stability} which is crucial to the proof of Lemma~\ref{lm:noisedecompbound}.
\begin{lemma}[\cite{liang2010trajectory}]\label{lm:poisregular}
Let Assumption~\ref{as:noise} be true. Then we have the following:
\begin{enumerate}[label=(\alph*)]
    \item For any $\theta\in\Theta$, the Markov kernel $P_\theta$ has a unique stationary distribution $\pi_\theta$. Moreover, $\nabla F(\theta,x):\Theta\times\mathbb{R}^d\to\Theta$ is measurable for all $\theta\in\Theta$, and $\expec{\nabla F(\theta,x)}{x\sim\pi_\theta}<\infty$.
    \item For any $\theta\in\Theta$, the solution to the Poisson equation~\eqref{eq:poissoneq} exists.  
    Furthermore, there exist a function $\mathcal{V}:\mathbb{R}^d\to[1,\infty)$ such that for all $\theta\in\Theta$, the following holds:
    \begin{enumerate}[label=(\roman*)]
        \item $\sup_{\theta\in\Theta}\norm{\nabla F(\theta,x)}_\mathcal{V}<\infty$,
        \item $\sup_{\theta\in\Theta}\left(\norm{u(\theta,x)}_\mathcal{V}+\norm{P_\theta u(\theta,x)}_\mathcal{V}\right)<\infty$,
        \item $\sup_{\theta\in\Theta}\left(\norm{u(\theta,x)-u(\theta',x)}_\mathcal{V}+\norm{P_\theta u(\theta,x)-P_{\theta'} u(\theta',x)}_\mathcal{V}\right)<\norm{\theta-\theta'}_2$.
    \end{enumerate}
\end{enumerate}
\end{lemma}
We now prove Lemma~\ref{lm:noisedecompbound}.
\vspace{0.1in}
\begin{proof}[Proof of Lemma~\ref{lm:noisedecompbound}]\label{pf:noisedecompbound}
	Let, 
	\begin{align}
	\begin{aligned}
	  &e_{k+1}=u(\theta_k,x_{k+1})-P_{\theta_k}u(\theta_k,x_{k})\\
	 &\nu_{k+1}=P_{\theta_{k+1}}u(\theta_{k+1},x_{k+1})-P_{\theta_{k}}u(\theta_{k},x_{k+1})+\frac{\eta_{k+2}-\eta_{k+1}}{\eta_{k+1}}P_{\theta_{k+1}}u(\theta_{k+1},x_{k+1})\\
	 &\tilde{\zeta}_{k+1}=\eta_{k+1}P_{\theta_{k}}u(\theta_{k},x_{k})\\
	 &\zeta_{k+1}=\frac{\tilde{\zeta}_{k+1}-\tilde{\zeta}_{k+2}}{\eta_{k+1}}. 
	\end{aligned}\label{eq:compdef}
	\end{align}
Now, one has, 
\begin{align*}
	\expec{e_{k+1}|\cF_{k}}{}=\expec{u(\theta_k,x_{k+1})|\cF_k}{}-P_{\theta_k}u(\theta_k,x_{k})=0
\end{align*}
We also have $\expec{\abs{e_{k+1}}}{}<\infty$. So $e_{k+1}$ is a martingale difference sequence. 
We also have, using Lemma~\ref{lm:poisregular}, and the fact that $\Theta$ is compact,
\begin{align*}
	\expec{\norm{\nu_{k+1}}_2}{}\leq c_1\norm{\theta_k-\theta_{k+1}}_2+c_2\eta_{k+2}\leq c_1\eta_{k+1}\norm{y_k-\theta_k}_2+c_2\eta_{k+2}\leq c_3\eta_{k+1}.
\end{align*}
Again, using Lemma~\ref{lm:poisregular}, we have
\begin{align*}
	\expec{\norm{\tilde{\zeta}_{k+1}}_2}{}\leq \eta_{k+1}\expec{\norm{P_{\theta_{k}}u(\theta_{k},x_{k})}_2}{}\leq c_4\eta_{k+1},
\end{align*}
where $c_i,\ i=1,2,3,4$ are constants. 
\end{proof}
\subsection{Proof of Theorem~\ref{th:mainthm}}
Before proving Theorem~\ref{th:mainthm}, we start with a few preliminaries. Let
\begin{align*}
    y_k'=\argmin_{y\in\Theta}\left\lbrace\left\langle z_k,y-\theta_k\right\rangle+\frac{\beta}{2}\norm{y-\theta_k}_2^2\right\rbrace,
\end{align*}
and let $\|y_k-y_k'\|_2\leq \delta_k$. 
We need the following result from \cite{jaggi2013revisiting} which bounds the distance between iterates generated by Algorithm~\ref{alg:icg}, $y_k$, and $y_k'$:
\begin{lemma}[\cite{jaggi2013revisiting}]
Under Assumption~\ref{aspt:constraint}, 
\begin{align*}
    \norm{y_k-y_k'}_2^2\leq \frac{4\mathcal{D}_\Theta^2(1+\omega)}{t_k+2},
\end{align*}
where $\omega$ is the accuracy of the LMO. 
\end{lemma}
\noindent Consider the following system:
\begin{align}
   & \tt_0=\theta_0 \quad \tz_0=z_0 \label{eq:tilde0}\\
    &\ty_k=\argmin_{y\in\Theta}\left\lbrace\left\langle \tz_k,y-\tt_k\right\rangle+\frac{\beta}{2}\norm{y-\tt_k}_2^2\right\rbrace \label{eq:tilde1}\\
    &\tt_{k+1}=\tt_k+\eta_{k+1}(\ty_k-\tt_k)\label{eq:tilde2}\\
    &\tz_{k+1}=z_{k+1}+\tilde{\zeta}_{k+1}.\label{eq:tilde3}
\end{align}
Equivalently one can also write:
\begin{align}
\ty_k=\Pi_{\Theta}\left(\tt_k-\frac{1}{\beta}\tz_k\right),\label{eq:tilde4}
\end{align}
where $\Pi_\Theta$ is the orthogonal projection on the set $\Theta$.
Let $\phi(\theta,z)$ be the following function:
\begin{align*}
    \phi(\theta,z)=\min_{y\in\Theta}\left(\left\langle z,y-\theta\right\rangle+\frac{\beta}{2}\|y-\theta\|_2^2\right).\numberthis\label{eq:phidef}
\end{align*}
By \cite[Lemma 4]{ghadimi2020single}, the function $\phi(\theta,z)$ has Lipschitz continuous gradient with Lipschitz constant $L_\phi$ for some $L_\phi>0$.
Define the merit function
\begin{align*}
    W(\theta,z)\coloneqq(f(\theta)-f^*)-\phi(\theta,z).\numberthis\label{eq:merit}
\end{align*}
Recall that as optimality measure we use the following:
\begin{align}\label{eq:vv}
    V(\theta_k,z_k)=\norm{\Pi_\Theta\left(\theta_k-\frac{z_k}{\beta}\right)-\theta_k}_2^2+\norm{z_k-\nabla f(\theta_k)}_2^2.
\end{align}

We now introduce a few intermediate results that are crucial in proving Theorem~\ref{th:mainthm}.  Specifically, in Lemma~\ref{lm:tildeorigdiff}, we show that the iterates generated by the auxiliary updates $\tt_k$ are close to the original updates $\theta_k$ of Algorithm~\ref{alg:asa}. Next, in Lemma~\ref{lm:vrelationorigtilde} we show that $V(\theta_k,z_k)$ is close to $V(\tt_k,\tz_k)$. 
Next, note that from ~\eqref{eq:vv}, we have 
$$
\sum_{k=1}^N\eta_kV(\tt_k,\tz_k) = \sum_{k=1}^N\eta_k\norm{\Pi_\Theta\left(\tt_k-\frac{\tz_k}{\beta}\right)-\tt_k}_2^2 + \sum_{k=1}^N\eta_k\norm{\tz_k-\nabla f(\tt_k)}_2^2
$$
We bound the first and the second term in the right hand side of the above equation in Lemma~\ref{lm:Vykthetakcompbound}, Lemma~\ref{lm:rk1bound} and \eqref{eq:gradfzwtdiffnormbound} respectively. 

\begin{lemma}\label{lm:tildeorigdiff}
Let the conditions of Lemma~\ref{lm:noisedecompbound} hold. Then, for $k\geq 1$, and for any $\gamma\in\mathbb{R}$, we have
\begin{align*}
    \expec{\norm{\tt_{k}-\theta_{k}}_2^2}{}\leq &\eta_{k}^3(1+\eta_{k}^{-\gamma})+2\sum_{i=1}^{k-1}\eta_{i}^3\left(1+\eta_i^{-\gamma}\right)\prod_{j=i+1}^{k}\left(1+\eta_{j}^{1+\gamma}\right)\\
    &+2\sum_{i=1}^{k-1}\eta_{i}\expec{\delta_{i-1}^2}{}\left(1+\eta_i^{-\gamma}\right)\prod_{j=i+1}^{k}\left(1+\eta_{j}^{1+\gamma}\right).\numberthis
\end{align*}
\end{lemma}
\begin{lemma}\label{lm:vrelationorigtilde}
Let the conditions of Lemma~\ref{lm:noisedecompbound} be true. Then, choosing $\eta_k=(N+k)^{-a}$, $a>1/2$, and setting $\gamma$ of Lemma~\ref{lm:tildeorigdiff} to $\gamma=1/a-1$, for $\delta_k\leq \eta_k$ we get
\begin{align*}
    \expec{V(\theta_k,z_k)}{}\leq 2\expec{V(\tt_k,\tz_k)}{}+(9+4L_G)\left(N^{1-4a}+8N^{2-4a}\right)+12N^{-2a}.
    \end{align*}
\end{lemma}
\begin{lemma}\label{lm:Vykthetakcompbound}
Let Assumption~\ref{aspt:constraint}, Assumption~\ref{as:contdiff}, and Assumption~\ref{as:noise} be true. Let $\{\tt_k,\tz_k,\ty_k,\}_{k\geq 0}$ be the sequence generated by \eqref{eq:tilde0}-\eqref{eq:tilde3}. Then $\forall k\geq 0$,
\begin{align}
    \frac{\beta}{2}\sum_{k=0}^{N-1}\eta_{k+1}\norm{\ty_k-\tt_k}_2^2\leq W(x_0,z_0)+\sum_{k=0}^{N-1}r_{k+1} \quad \forall N\geq 1,\label{eq:ykthetakdiffnormbound}
\end{align}
where for $k\geq 0$, and
\begin{align*}
    r_{k+1}= \frac{(L_G+L_\phi)\eta_{k+1}^2}{2}\norm{\ty_k-\tt_k}_2^2+\frac{L_\phi}{2}\norm{\tz_{k+1}-\tz_k}_2^2+\eta_{k+1}\inprod{\tt_k-\ty_k,\tilde{\epsilon}_{k+1}}+\frac{\eta_{k+1}L_G^2}{\beta}\norm{\theta_k-\tt_k}_2^2.
\end{align*}
\end{lemma}
\begin{lemma}\label{lm:rk1bound}
Let $\{\tt_k,\tz_k,\ty_k,\}_{k\geq 0}$ be the sequence generated by \eqref{eq:tilde0}-\eqref{eq:tilde3}, and Assumption~\ref{aspt:constraint}-\ref{as:noise} hold. Then,
\begin{enumerate}
    \item If $\eta_0=1$, we have,
    \begin{align}
        \beta^2\expec{\norm{\ty_k-\tt_k}_2^2|\cF_{k-1}}{}\leq \expec{\norm{\tz_k}_2^2|\cF_{k-1}}{}\leq \sigma^2 \quad \forall k\geq 1; \label{eq:expeczkbound}
    \end{align}
    \item If $\eta_k\leq 1/\sqrt{2}$ for all $k\geq 1$, then,
    \begin{align}
        \sum_{k=0}^\infty \expec{\norm{\tz_{k+1}-\tz_k}_2^2|\cF_k}{}&\leq 2\left(\|\tz_0\|_2^2+24\sigma^2\sum_{k=0}^\infty\eta_k^2\right),\\
        \sum_{k=0}^\infty\expec{r_{k+1}|\cF_k}{}&\leq \sigma_3^2\sum_{k=0}^\infty\eta_k^2+\frac{L_G^2}{\beta}\sum_{k=0}^\infty\eta_{k+1}\expec{\norm{\theta_k-\tt_k}_2^2}{},\label{eq:rksumbound}
    \end{align}
    where
    \begin{align*}
        \sigma_3^2=\frac{1}{2}\left((3L_G+L_\phi)\frac{\sigma^2}{\beta^2}+4L_\phi(\|z_0\|_2^2+24\sigma^2)+2\right).
    \end{align*}
\end{enumerate}
\end{lemma}
The proofs of Lemma~\ref{th:mainthm}, Lemma~\ref{lm:tildeorigdiff}, Lemma~\ref{lm:Vykthetakcompbound}, and Lemma~\ref{lm:rk1bound} are provided in Section~\ref{sec:proofoflemmaa3toa6}. We now prove Theorem~\ref{th:mainthm}.
\begin{proof}[Proof of Theorem~\ref{th:mainthm}]\label{pf:mainthm}
Define,
\begin{align*}
    \Gamma_1\coloneqq
    1 \qquad \Gamma_k\coloneqq \prod_{i=0}^{k-1}(1-\eta_{i+1}) \quad \forall k\geq 2.\numberthis\label{eq:gammadef}
\end{align*}
Using the update of Algorithm~\ref{alg:asa} we have
\begin{align*}
    \nabla f(\tt_{k+1})-\tz_{k+1}=&~(1-\eta_{k+1})\left(\nabla f(\tt_{k})-\tz_k+\nabla f(\tt_{k+1})-\nabla f(\tt_{k})\right)\\
    &~+\eta_{k+1}\left(\nabla f(\tt_{k+1})-\nabla f(\tt_{k})-\tilde{\epsilon}_{k+1}\right)+\eta_{k+1}\left(\nabla f(\tt_{k})-\nabla f(\theta_{k})\right).
\end{align*}
Dividing both sides of the above equation by $\Gamma_{k+1}$ we obtain
\begin{align*}
    \frac{\nabla f(\tt_{k+1})-\tz_{k+1}}{\Gamma_{k+1}}=& \frac{1}{\Gamma_k}\left(\nabla f(\tt_{k})-\tz_k+\nabla f(\tt_{k+1})-\nabla f(\tt_{k})\right)+\frac{\eta_{k+1}}{\Gamma_{k+1}}\left(\nabla f(\tt_{k+1})-\nabla f(\tt_{k})-\tilde{\epsilon}_{k+1}\right)\\
    &~~~~+\frac{\eta_{k+1}}{\Gamma_{k+1}}\left(\nabla f(\tt_{k})-\nabla f(\theta_{k})\right)\\
    =&\frac{1}{\Gamma_k}\left(\nabla f(\tt_{k})-\tz_k\right)+\frac{1}{\Gamma_{k+1}}\left(\nabla f(\tt_{k+1})-\nabla f(\tt_{k})\right)-\frac{\eta_{k+1}}{\Gamma_{k+1}}\left(\tilde{\epsilon}_{k+1}+\nabla f(\theta_{k})-\nabla f(\tt_{k})\right)
\end{align*}
Summing both sides from $k=1$ to $k=i-1$ we obtain
\begin{align*}
    \nabla f(\tt_{i})-\tz_{i}
    =\sum_{k=0}^{i-1}\frac{\Gamma_{i}}{\Gamma_{k+1}}\left(\nabla f(\tt_{k+1})-\nabla f(\tt_{k})\right)-\sum_{k=0}^{i-1}\frac{\eta_{k+1}\Gamma_{i}}{\Gamma_{k+1}}\left(\tilde{\epsilon}_{k+1}+\nabla f(\theta_{k})-\nabla f(\tt_{k})\right).
\end{align*}
Hence,
\begin{align*}
    \nabla f(\tt_{i})-\tz_{i}=&\frac{\Gamma_i}{\Gamma_{i-1}}\left(\nabla f(\tt_{i-1})-\tz_{i-1}\right)+\left(\nabla f(\tt_i)-\nabla f(\tt_{i-1})\right)-\eta_i\left(\tilde{\epsilon}_{i}+\nabla f(\theta_{i-1})-\nabla f(\tt_{i-1})\right)\\
    =&(1-\eta_{i})\left(\nabla f(\tt_{i-1})-\tz_{i-1}\right)+\frac{\eta_i}{\eta_i}\left(\nabla f(\tt_i)-\nabla f(\tt_{i-1})\right)-\eta_i\left(\tilde{\epsilon}_{i}+\nabla f(\theta_{i-1})-\nabla f(\tt_{i-1})\right).
\end{align*}
Using Young's inequality and Jensen's inequality, 
\begin{align*}
   \norm{\nabla f(\tt_{i})-\tz_{i}}_2^2\leq &\frac{1-\eta_i/4}{1-\eta_i/2}\norm{(1-\eta_{i})\left(\nabla f(\tt_{i-1})-\tz_{i-1}\right)+\frac{\eta_i}{\eta_i}\left(\nabla f(\tt_i)-\nabla f(\tt_{i-1})\right)-\eta_i\tilde{\epsilon}_{i}}_2^2\\
   &\qquad+\frac{4-\eta_i}{\eta_i}\eta_i^2\norm{\nabla f(\theta_{i-1})-\nabla f(\tt_{i-1})}_2^2\\
    \leq &\frac{1-\eta_i/4}{1-\eta_i/2} \cdot I_1 +4L_G^2\eta_i\norm{\theta_{i-1}-\tt_{i-1}}_2^2,\numberthis\label{eq:gradfzdiffnormintermed}
\end{align*}
where 
\begin{align*}
    I_1\coloneqq&(1-\eta_{i}) \norm{\nabla f(\tt_{i-1})-\tz_{i-1}}_2^2+\frac{\norm{\nabla f(\tt_i)-\nabla f(\tt_{i-1})}_2
    ^2}{\eta_i}+ \eta_i^2\norm{\tilde{\epsilon}_i}_2^2\\
    &-2\eta_i\left\langle (1-\eta_{i})\left(\nabla f(\tt_{i-1})-\tz_{i-1}\right)+\left(\nabla f(\tt_i)-\nabla f(\tt_{i-1})\right),\tilde{\epsilon}_i\right\rangle.
\end{align*}
Taking conditional expectation of $I_1$ with respect to $\cF_{i-1}$, using \eqref{eq:lipgrad}, Assumption~\ref{as:noise}, and \eqref{eq:tilde2}, we then obtain
\begin{align*}
    \expec{I_1 \mid\cF_{i-1}}{}\leq & (1-\eta_i)\norm{\nabla f(\tt_{i-1})-\tz_{i-1}}_2^2+\eta_i L_G^2\norm{\ty_{i-1}-\tt_{i-1}}_2
    ^2+ \eta_i^2\sigma^2\\
    &-2\eta_i\expec{\left\langle (1-\eta_{i})\left(\nabla f(\tt_{i-1})-\tz_{i-1}\right)+\left(\nabla f(\tt_i)-\nabla f(\tt_{i-1})\right),\tilde{\zeta}_i\right\rangle |\cF_{i-1} }{}\\
    \leq & (1-\eta_i)\norm{\nabla f(\tt_{i-1})-\tz_{i-1}}_2^2+\eta_i L_G^2\expec{\norm{\ty_{i-1}-\tt_{i-1}}_2
    ^2|\cF_{i-1}}{}+ \eta_i^2\sigma^2\\
    &+2\eta_i^2(1-\eta_i)^2\norm{\nabla f(\tt_{i-1})-\tz_{i-1}}_2^2+2\eta_i^4L_G^2\norm{\ty_{i-1}-\tt_{i-1}}_2^2+\eta_i^2\\
    \leq & \left(1-\frac{\eta_i}{2}\right)\norm{\nabla f(\tt_{i-1})-\tz_{i-1}}_2^2+2\eta_i L_G^2\expec{\norm{\ty_{i-1}-\tt_{i-1}}_2
    ^2|\cF_{i-1}}{}+ \eta_i^2(1+\sigma^2).
\end{align*}
Now, taking expectation on both sides of \eqref{eq:gradfzdiffnormintermed},
\begin{align*}
    \expec{\norm{\nabla f(\tt_{i})-\tz_{i}}_2^2 }{}&\leq  \left(1-\frac{\eta_i}{4}\right)\expec{\norm{\nabla f(\tt_{i-1})-\tz_{i-1}}_2^2}{}+4\eta_i L_G^2\expec{\norm{\ty_{i-1}-\tt_{i-1}}_2
    ^2}{}\\
    &~~+ 2\eta_i^2(1+\sigma^2)+4L_G^2\eta_i\expec{\norm{\theta_{i-1}-\tt_{i-1}}_2^2}{}\\
    &\leq  Y_0^{i}\expec{\norm{\nabla f(\tt_{0})-\tz_{0}}_2^2}{}+4 L_G^2\sum_{k=1}^{i}Y_k^i\eta_k\expec{\norm{\ty_{k-1}-\tt_{k-1}}_2
    ^2}{}
    +2\sum_{k=1}^i Y_{k-1}^i\eta_k^2(1+\sigma^2)\\
    &~~+4L_G^2\sum_{k=1}^iY_{k-1}^i\eta_k\expec{\norm{\theta_{k-1}-\tt_{k-1}}_2^2}{},
\end{align*}
where 
\begin{align*}
    Y_i^i=1~~ \text{for all $i$, \qquad and} \qquad Y_k^i=\prod_{j=k+1}^i\left(1-\frac{\eta_j}{4}\right) \ \text{ for } i> k.\numberthis\label{eq:Ykidef}
\end{align*}
Then,
\begin{align*}
    \sum_{i=1}^{N}\eta_i\expec{\norm{\nabla f(\tt_{i})-\tz_{i}}_2^2 }{}\leq& \sum_{i=1}^{N}\eta_iY_0^i\expec{\norm{\nabla f(\tt_{0})-\tz_{0}}_2^2}{}+4 L_G^2\sum_{i=1}^{N}\sum_{k=1}^{i}Y_k^i\eta_i\eta_k\expec{\norm{\ty_{k-1}-\tt_{k-1}}_2
    ^2}{}\\
    &+2\sum_{i=1}^{N}\sum_{k=1}^i Y_{k-1}^i\eta_i\eta_k^2(1+\sigma^2)+4L_G^2\sum_{i=1}^{N}\sum_{k=1}^iY_{k-1}^i\eta_i\eta_k\expec{\norm{\theta_{k-1}-\tt_{k-1}}_2^2}{}\\
    =& \expec{\norm{\nabla f(\tt_{0})-\tz_{0}}_2^2}{}+4 L_G^2\sum_{k=1}^{N}\sum_{i=k}^{N}Y_k^i\eta_i\eta_k\expec{\norm{\ty_{k-1}-\tt_{k-1}}_2
    ^2}{}\\
    &+2\sum_{k=1}^{N}\sum_{i=k}^{N} Y_{k-1}^i\eta_i\eta_k^2(1+\sigma^2)+4L_G^2\sum_{k=1}^{N}\sum_{i=k}^NY_{k-1}^i\eta_i\eta_k\expec{\norm{\theta_{k-1}-\tt_{k-1}}_2^2}{}\\
    \leq & \expec{\norm{\nabla f(\tt_{0})-\tz_{0}}_2^2}{}+4 L_G^2\sum_{k=0}^{N-1}\eta_k\expec{\norm{\ty_{k}-\tt_{k}}_2
    ^2}{}\\
    &+2\sum_{k=1}^{N}\eta_k^2(1+\sigma^2)+4L_G^2\sum_{k=0}^{N-1}\eta_{k+1}\expec{\norm{\theta_{k}-\tt_{k}}_2^2}{}.
\end{align*}
The last inequality follows by Lemma~\ref{lm:Ykibound}.
Combining \eqref{eq:ykthetakdiffnormbound}, and \eqref{eq:rksumbound}, we get, 
 \begin{align*}
    \sum_{i=1}^{N}\eta_i\expec{\norm{\nabla f(\tt_{i})-\tz_{i}}_2^2 }{}\leq& ~~\expec{\norm{\nabla f(\tt_{0})-\tz_{0}}_2^2}{}+\sum_{k=1}^{N}\eta_k^2(1+\sigma^2)+4L_G^2\sum_{k=0}^{N-1}\eta_{k+1}\expec{\norm{\theta_{k}-\tt_{k}}_2^2}{}\\
    &+4 L_G^2\left(W(x_0,z_0)+\sigma^2\sum_{k=0}^N\eta_k^2+\frac{L_G^2}{\beta}\sum_{k=0}^\infty\eta_{k+1}\expec{\norm{\theta_k-\tt_k}_2^2}{}\right).\numberthis\label{eq:gradfzwtdiffnormbound}
\end{align*}
Combining \eqref{eq:gradfzwtdiffnormbound}, and \eqref{eq:ykthetakdiffnormbound}, we get,
\begin{align*}
    \left(\sum_{k=1}^N\eta_k \right) \expec{V(\tt_k,\tz_k)}{}    \leq& \expec{\norm{\nabla f(\tt_{0})-\tz_{0}}_2^2}{}+4L_G^2W(x_0,z_0)+2 \sum_{k=1}^{N}\eta_k^2(1+\sigma^2)\\
    &~+(4L_G^2+4L_G^4/\beta)\sum_{k=0}^{N-1}\eta_{k+1}\expec{\norm{\theta_{k}-\tt_{k}}_2^2}{}.\numberthis\label{eq:Vtttzintermedbound}
\end{align*}
Choosing $\eta_k=(N+k)^{-a}$, using Lemma~\ref{lm:tildeorigdiff}, for $\gamma=1/a-1$, we get,
\begin{align*}
    \sum_{k=0}^{N-1}\eta_{k+1}\expec{\norm{\theta_{k}-\tt_{k}}_2^2}{}\leq& \sum_{k=0}^{N-1}\eta_{k+1}\eta_{k}^3(1+\eta_{k}^{-\gamma})+\sum_{k=0}^{N-1}\eta_{k+1}\sum_{i=1}^{k-1}\eta_{i}^3\left(1+\eta_i^{-\gamma}\right)\prod_{j=i+1}^{k}\left(1+\eta_{j}^{1+\gamma}\right)\\
    \leq& 2\sum_{k=0}^{N-1}(N+k+1)^{-a}\sum_{i=1}^{k-1}(N+i)^{-3a}\left(1+(N+i)^{a\gamma}\right)\prod_{j=i+1}^{k}\left(1+(N+j)^{-a(1+\gamma)}\right)\\
    \leq & 2\sum_{k=0}^{N-1}N^{-4a}\sum_{i=1}^{k-1}\left(1+(2N)^{1-a}\right)\left(1+N^{-1}\right)^{k-i}\\
    \leq & 2\sum_{k=0}^{N-1}N^{1-4a}\left(1+(2N)^{1-a}\right)\left((1+N^{-1})^k-1\right)\\
    \leq &8N^{2-5a}\left(N(1+N^{-1})^N-2N\right)\\
    \leq &8N^{3-5a}. 
\end{align*}
Hence, we have
\begin{align*}
    &\frac{\sum_{k=0}^{N-1}\eta_{k+1}\expec{\norm{\theta_{k}-\tt_{k}}_2^2}{}}{\sum_{k=1}^N\eta_k}\leq \frac{8N^{3-5a}}{\sum_{k=1}^N(2N)^{-a}}\leq 16N^{2-4a}.\numberthis\label{eq:thetattdiffwtnorm2}
\end{align*}
Using \eqref{eq:thetattdiffwtnorm2}, and \eqref{eq:Vtttzintermedbound}, we get
\begin{align*}
    \expec{V(\tt_k,\tz_k)}{}\leq \left(\expec{\norm{\nabla f(\tt_{0})-\tz_{0}}_2^2}{}+4 L_G^2W(x_0,z_0)\right)N^{a-1}+2(1+\sigma^2)N^{-a}+16N^{2-4a}.
\end{align*}
Then using Lemma~\ref{lm:vrelationorigtilde}, we get,
\begin{align*}
    \expec{V(\theta_k,z_k)}{}\leq &\left(\expec{\norm{\nabla f(\tt_{0})-\tz_{0}}_2^2}{}+4L_G^2W(x_0,z_0)\right)N^{a-1}\\
    &+2(1+\sigma^2)N^{-a}+ (9+4L_G)\left(N^{1-4a}+8N^{2-4a}\right)+12N^{-2a}+16N^{2-4a}.\numberthis\label{eq:vtttzbound}
\end{align*}
Now choosing, $a=3/5$, we get,
\begin{align*}
    \expec{V(\theta_k,z_k)}{}\leq &\left(\expec{\norm{\nabla f(\tt_{0})-\tz_{0}}_2^2}{}+4L_G^2W(x_0,z_0)\right)N^{-2/5}\\
    &+2(1+\sigma^2)N^{-3/5}+ (9+4L_G)\left(N^{-7/5}+8N^{-2/5}\right)+12N^{-6/5}+16N^{-2/5}\\
    =& \mathcal{O}\left(N^{-\frac{2}{5}}\right).
\end{align*}
\end{proof}

\noindent Now we the provide the proofs of the Lemmas required to prove Theorem~\ref{th:mainthm}.

\subsection{Proof of Lemmas related to Theorem~\ref{th:mainthm}}\label{sec:proofoflemmaa3toa6}
\begin{proof}[Proof of Lemma~\ref{lm:tildeorigdiff}]
By Jensen's inequality, contraction property of the projection operator, and Young's inequality, we get
\begin{align*}
    \norm{\tt_{k+1}-\theta_{k+1}}_2^2
    \leq& (1-\eta_{k+1})\norm{\tt_{k}-\theta_{k}}_2^2+\eta_{k+1}\norm{\ty_k-y_k}_2^2\\
    \leq & (1-\eta_{k+1})\norm{\tt_{k}-\theta_{k}}_2^2+\eta_{k+1}\left(\norm{\tt_k-\theta_k}_2+\norm{\tz_k/\beta-z_k/\beta}_2+\delta_k\right)^2\\
    \leq &(1-\eta_{k+1})\norm{\tt_{k}-\theta_{k}}_2^2+\eta_{k+1}(1+\eta_{k+1}^\gamma)\norm{\tt_k-\theta_k}_2^2\\
    &~~+2\eta_{k+1}(1+\eta_{k+1}^{-\gamma})\norm{\tz_k/\beta-z_k/\beta}_2^2+2\eta_{k+1}(1+\eta_{k+1}^{-\gamma})\delta_k^2.
\end{align*}
Now taking expectation on both sides, and using Lemma~\ref{lm:noisedecompbound}, we have
\begin{align*}
   \expec{ \norm{\tt_{k+1}-\theta_{k+1}}_2^2}{}
   \leq&  \left(1+\eta_{k+1}^{1+\gamma}\right)\expec{\norm{\tt_{k}-\theta_{k}}_2^2}{}+2\eta_{k+1}^3(1+\eta_{k+1}^{-\gamma})+2\eta_{k+1}(1+\eta_{k+1}^{-\gamma})\expec{\delta_k^2}{}\\
   \leq & \eta_{k+1}^3(1+\eta_{k+1}^{-\gamma})+2\sum_{i=1}^{k}\eta_{i}^3\left(1+\eta_i^{-\gamma}\right)\prod_{j=i+1}^{k+1}\left(1+\eta_{j}^{1+\gamma}\right)\\
   &~~+2\sum_{i=1}^{k}\eta_{i}\expec{\delta_{i-1}^2}{}\left(1+\eta_i^{-\gamma}\right)\prod_{j=i+1}^{k+1}\left(1+\eta_{j}^{1+\gamma}\right).
\end{align*}
\end{proof}
\begin{proof}[Proof of Lemma~\ref{lm:vrelationorigtilde}]
Using \eqref{eq:lipgrad}, and contraction property of the projection operator,
\begin{align*}
    V(\theta_k,z_k)=&\norm{\Pi_{\Theta}\left(\theta_k-z_k\right)-\theta_k}_2^2+\norm{z_k-\nabla f(\theta_k)}_2^2\\
    \leq &2\norm{\Pi_{\Theta}\left(\theta_k-z_k\right)-\theta_k-\Pi_{\Theta}\left(\tt_k-\tz_k\right)+\tt_k}_2^2+2\norm{\tz_k-\nabla f(\tt_k)-z_k+\nabla f(\theta_k)}_2^2\\
    &+2\norm{\Pi_{\Theta}\left(\tt_k-\tz_k\right)-\tt_k}_2^2+2\norm{\tz_k-\nabla f(\tt_k)}_2^2\\
    \leq& 2V(\tt_k,\tz_k)+(8+4L_G)\norm{\theta_k-\tt_k}_2^2+12\norm{z_k-\tz_k}_2^2.
\end{align*}
Using Lemma~\ref{lm:tildeorigdiff}, and Lemma~\ref{lm:noisedecompbound}, we get,
\begin{align*}
    \expec{V(\theta_k,z_k)}{}\leq& 2\expec{V(\tt_k,\tz_k)}{}+(8+4L_G)\expec{\norm{\theta_k-\tt_k}_2^2}{}+12\expec{\norm{z_k-\tz_k}_2^2}{}\\
    \leq & 2\expec{V(\tt_k,\tz_k)}{}+(8+4L_G)\left(\eta_{k}^3(1+\eta_{k}^{-\gamma})+2\sum_{i=1}^{k-1}\eta_{i}^3\left(1+\eta_i^{-\gamma}\right)\prod_{j=i+1}^{k}\left(1+\eta_{j}^{1+\gamma}\right)\right.\\
    &\left.+2\sum_{i=1}^{k-1}\eta_{i}\expec{\delta_{i-1}^2}{}\left(1+\eta_i^{-\gamma}\right)\prod_{j=i+1}^{k}\left(1+\eta_{j}^{1+\gamma}\right)\right)+12\eta_{k+1}^2.
\end{align*}
For $\delta_{k-1}\leq \eta_k$, choosing $\eta_k=\frac{1}{(N+k)^a}$ with $a>1/2$, we get,
\begin{align*}
    \expec{V(\theta_k,z_k)}{}\leq & 2\expec{V(\tt_k,\tz_k)}{}+\frac{12}{(N+k+1)^{2a}}\\
    &~+(8+4L_G)\left(\frac{1+(N+k)^{a\gamma}}{(N+k)^{3a}}+4\sum_{i=1}^{k-1}\frac{1+(N+i)^{a\gamma}}{(N+i)^{3a}}\prod_{j=i+1}^{k}\left(1+(N+j)^{-a(1+\gamma)}\right)\right)\\
    \leq & 2\expec{V(\tt_k,\tz_k)}{}+(9+4L_G)\left(\frac{1}{N^{3a-a\gamma}}+\sum_{i=1}^{k-1}\frac{4}{N^{3a-a\gamma}}\left(1+\frac{1}{N^{a(1+\gamma)}}\right)^i\right)+\frac{12}{N^{2a}}\\
    \leq & 2\expec{V(\tt_k,\tz_k)}{}+(9+4L_G)\left(\frac{1}{N^{3a-a\gamma}}+\frac{4}{N^{2a-2a\gamma}}\left[\left(1+\frac{1}{N^{a(1+\gamma)}}\right)^N-1\right]\right)+\frac{12}{N^{2a}}\\
    \leq & 2\expec{V(\tt_k,\tz_k)}{}+(9+4L_G)\left(\frac{1}{N^{3a-a\gamma}}+\frac{4}{N^{2a-2a\gamma}}\left[\exp\left(N^{1-a(1+\gamma)}\right)-1\right]\right)+\frac{12}{N^{2a}}\\
    \leq & 2\expec{V(\tt_k,\tz_k)}{}+(9+4L_G)\left(\frac{1}{N^{3a-a\gamma}}+8N^{1-3a+a\gamma}\right)+\frac{12}{N^{2a}}.\numberthis\label{eq:thetattdiffnorm2bound}
\end{align*}
Setting $\gamma=1/a-1$, we get,
\begin{align*}
    &\expec{V(\theta_k,z_k)}{}
    \leq  2\expec{V(\tt_k,\tz_k)}{}+(9+4L_G)\left(N^{1-4a}+8N^{2-4a}\right)+12N^{-2a}.
\end{align*}
\end{proof}

\begin{proof}[Proof of Lemma~\ref{lm:Vykthetakcompbound}]\label{pf:Vykthetakcompbound}

Recall that,
\begin{align*}
    \phi(\theta,z)=\min_{y\in\Theta}\left(\left\langle z,y-\theta\right\rangle+\frac{\beta}{2}\|y-\theta\|_2^2\right).\numberthis\label{eq:phidef}
\end{align*}
It is easy to verify that $\phi(\theta,z)$ has a $L_\phi$-Lipschitz continuous gradient \cite[Lemma 3]{ghadimi2020single} where $$L_\phi=2\sqrt{(1+\beta)^2+(1+1/(2\beta))^2}.$$
Using the definition of $\phi(\theta,z)$ in \eqref{eq:phidef}, and Lipschitz continuity of its gradient, we have 
\begin{align*}
    \phi(\tt_k,\tz_k)-\phi(\tt_{k+1},\tz_{k+1})
    \leq& \inprod{\tz_k+\beta(\ty_k-\tt_k),\tt_{k+1}-\tt_k}-\inprod{\ty_k-\tt_k,\tz_{k+1}-\tz_k}\\&~+\frac{L_\phi}{2}\left[\norm{\tt_{k+1}-\tt_k}_2^2+\norm{\tz_{k+1}-\tz_k}_2^2\right].\numberthis\label{eq:philipgrad}
\end{align*}
By the optimality condition of the subproblem \eqref{eq:tilde1} we have,
\begin{align*}
    \inprod{\tz_k+\beta(\ty_k-\tt_k),y-\ty_k}\geq 0 \quad \forall y\in\Theta.\numberthis\label{eq:subopt1}
\end{align*}
For $y=\tt_k$ we have,
\begin{align*}
    \inprod{\tz_k+\beta(\ty_k-\tt_k),\ty_k-\tt_k}\leq 0.\numberthis\label{eq:subopt2}
\end{align*}
Note that this also implies 
\begin{align*}
    \phi(\tt_k,\tz_k)\leq 0.\numberthis\label{eq:phipos}
\end{align*}
We also have,
\begin{align*}
    \tz_{k+1}-\tz_k=& \tz_{k+1}-(1-\eta_{k+1})\tz_k-\eta_{k+1}\tz_k\\
    =& z_{k+1}-(1-\eta_{k+1})z_k-\eta_{k+1}\tz_k+\tilde{\zeta}_{k+2}-(1-\eta_{k+1})\tilde{\zeta}_{k+1}\\
    =& \eta_{k+1}\left(\nabla f(\theta_k)+e_{k+1}+\nu_{k+1}+\zeta_{k+1}\right)-\eta_{k+1}\tz_k+\tilde{\zeta}_{k+2}-(1-\eta_{k+1})\tilde{\zeta}_{k+1}\\
    =& \eta_{k+1}\left(\nabla f(\theta_k)+e_{k+1}+\nu_{k+1}\right)+(\tilde{\zeta}_{k+1}-\tilde{\zeta}_{k+2})-\eta_{k+1}\tz_k+\tilde{\zeta}_{k+2}-(1-\eta_{k+1})\tilde{\zeta}_{k+1}\\
    =& \eta_{k+1}\left(\nabla f(\theta_k)+\tilde{\epsilon}_{k+1}\right)-\eta_{k+1}\tz_k,
\end{align*}
where,
$\tilde{\epsilon}_{k}=e_{k}+\nu_k+\tilde{\zeta}_k$.
Then, using \eqref{eq:lipgrad} we have,
\begin{align*}
    \inprod{\ty_k-\tt_k,\tz_{k+1}-\tz_k}=& \inprod{\ty_k-\tt_k,\eta_{k+1}\left(\nabla f(\theta_k)+\tilde{\epsilon}_{k+1}\right)-\eta_{k+1}\tz_k}\\
    =& \inprod{\tt_{k+1}-\tt_k,\nabla f(\tt_k)}+\inprod{\tt_{k+1}-\tt_k,\nabla f(\theta_k)-\nabla f(\tt_k)}\\
    &~+\inprod{\ty_k-\tt_k,\eta_{k+1}\tilde{\epsilon}_{k+1}}-\inprod{\tt_{k+1}-\tt_k,\tz_k}\\
    \geq & f(\tt_{k+1})-f(\tt_k)-\frac{L_G}{2}\|\tt_{k+1}-\tt_k\|_2^2-\frac{\beta}{2\eta_{k+1}}\norm{\tt_{k+1}-\tt_k}_2^2\\
    &~+\inprod{\ty_k-\tt_k,\eta_{k+1}\tilde{\epsilon}_{k+1}}-\inprod{\tt_{k+1}-\tt_k,\tz_k}.\numberthis\label{eq:ythetazk1zkinprod}
\end{align*}
Combining \eqref{eq:philipgrad}, \eqref{eq:subopt1}, \eqref{eq:subopt2}, and \eqref{eq:ythetazk1zkinprod}, using \eqref{eq:lipgrad}, and rearranging, we get,
\begin{align*}
    &\phi(\tt_k,\tz_k)-\phi(\tt_{k+1},\tz_{k+1})\\
    \leq & f(\tt_k)-f(\tt_{k+1})+\frac{L_G}{2}\|\tt_{k+1}-\tt_k\|_2^2+\frac{\beta}{2\eta_{k+1}}\norm{\tt_{k+1}-\tt_k}_2^2+\frac{\eta_{k+1}}{\beta}\norm{\nabla f(\theta_k)-\nabla f(\tt_k)}_2^2\\
    &-\inprod{\ty_k-\tt_k,\eta_{k+1}\tilde{\epsilon}_{k+1}}-\eta_{k+1}\beta\norm{\ty_k-\tt_k}_2^2
    +\frac{L_\phi}{2}\left[\norm{\tt_{k+1}-\tt_k}_2^2+\norm{\tz_{k+1}-\tz_k}_2^2\right]\\
    &W(\tt_{k+1},\tz_{k+1})-W(\tt_{k},\tz_{k})\\
    \leq&  -\frac{\eta_{k+1}\beta}{2}\norm{\ty_k-\tt_k}_2^2+\frac{(L_G+L_\phi)\eta_{k+1}^2}{2}\norm{\ty_k-\tt_k}_2^2+\frac{L_\phi}{2}\norm{\tz_{k+1}-\tz_k}_2^2+\frac{\eta_{k+1}L_G^2}{\beta}\norm{\theta_k-\tt_k}_2^2\\
    &-\eta_{k+1}\inprod{\ty_k-\tt_k,\tilde{\epsilon}_{k+1}}
\end{align*}
Summing both sides from $k=0$ to $N-1$, and using \eqref{eq:phipos}, we get,
\begin{align*}
    \sum_{k=0}^{i}\frac{\eta_{k+1}\beta}{2}\norm{\ty_k-\tt_k}_2^2\leq W(\tt_0,\tz_0)+\sum_{k=0}^{N-1}r_{k+1},
\end{align*}
where 
\begin{align*}
    r_{k+1}=\frac{(L_G+L_\phi)\eta_{k+1}^2}{2}\norm{\ty_k-\tt_k}_2^2+\frac{L_\phi}{2}\norm{\tz_{k+1}-\tz_k}_2^2
    +\eta_{k+1}\inprod{\tt_k-\ty_k,\tilde{\epsilon}_{k+1}}+\frac{\eta_{k+1}L_G^2}{\beta}\norm{\theta_k-\tt_k}_2^2.
\end{align*}
\end{proof}

\begin{proof}[Proof of Lemma~\ref{lm:rk1bound}]\label{pf:rk1bound}
The proof is similar to that of \cite[Proposition 1]{ghadimi2020single} with minor modifications. First, note that we no longer have $\expec{(\tt_k-\ty_k)^\top\tilde\epsilon_{k}|\cF_{k-1}}{}=0$ since $\{\tilde\epsilon_{k}\}_k$ is no longer a martingale difference sequence. But we can show that the term is small enough, i.e., of the order of the stepsize. Indeed note that, using \eqref{eq:expeczkbound}, we have
\begin{align*}
    \expec{(\tt_k-\ty_k)^\top\tilde\epsilon_{k}|\cF_{k-1}}{}=&\expec{(\tt_k-\ty_k)^\top(\nu_k+\tilde\zeta_{k})|\cF_{k-1}}{}\\
    \leq& \sqrt{\expec{\|\tt_k-\ty_k\|_2^2|\cF_{k-1}}{}}\sqrt{\expec{\|\nu_k+\tilde\zeta_{k}\|_2^2|\cF_{k-1}}{}}\\
    \leq &  \sqrt{\expec{\frac{2\|\tz_k\|_2^2}{\beta^2}|\cF_{k-1}}{}}\eta_k\\
    \leq &\frac{2\sigma\eta_k}{\beta}.\numberthis\label{eq:interactionbound}
\end{align*}
Combining \eqref{eq:interactionbound} with the proof of \cite[Proposition 1]{ghadimi2020single} we obtain the result in Lemma~\ref{lm:rk1bound}.
\end{proof}

\begin{lemma}\label{lm:Ykibound}
Let $Y_k^i$ be defined as in \eqref{eq:Ykidef} and  
let $\eta_j$ be of the form $(N+j)^{-a}$ where $1/2<a<1$. Then, for $k\leq i\leq N$, we have
\begin{align*}
   & Y_{k-1}^i\leq \exp\left(-\frac{1}{8}\left((N+i)^{1-a}-(N+k)^{1-a}\right)\right)\\
   &\sum_{i=k}^NY^i_{k-1}\eta_i=\mathcal{O}(1).\numberthis\label{eq:Ykibound}
\end{align*}
\end{lemma}
\begin{proof}
First, using the fact that $1-x\leq \exp(-x)$, we obtain
\begin{align*}
 Y_{k-1}^i=\prod_{j=k}^i\left(1-\frac{\eta_j}{4}\right)\leq& \prod_{j=k}^i\exp\left(-\frac{\eta_j}{4}\right)= \exp\left(-\sum_{j=k}^i\frac{(N+j)^{-a}}{4}\right)\\
 \leq & \exp\left(-\int_{k}^i\frac{(N+j)^{-a}}{8}\right)dj\\
 \leq & \exp\left(-\frac{1}{8}\left((N+i)^{1-a}-(N+k)^{1-a}\right)\right),
\end{align*}
which proves the first claim. Next, we have
\begin{align*}
    \sum_{i=k}^NY^i_{k-1}\eta_i\leq & \sum_{i=k}^N\exp\left(-\frac{1}{8}\left((N+i)^{1-a}-(N+k)^{1-a}\right)\right)(N+i)^{-a}\\
    =& \exp\left(\frac{(N+k)^{1-a}}{8}\right)\sum_{i=k+N}^{2N}\exp\left(-\frac{i^{1-a}}{8}\right)i^{-a}\\
    \leq & \exp\left(\frac{(N+k)^{1-a}}{8}\right)\int_{k+N}^{2N}\exp\left(-\frac{(i-1)^{1-a}}{8}\right)(i-1)^{-a}di\\
    = & \frac{\exp\left(\frac{(N+k)^{1-a}}{8}\right)}{1-a}\int_{(k+N-1)^{1-a}}^{(2N-1)^{1-a}}\exp\left(-u\right)du\\
    \leq & \frac{\exp\left(\frac{(N+k)^{1-a}}{8}-\frac{(N+k-1)^{1-a}}{8}\right)}{1-a}\\
    \leq & \frac{e}{1-a},
\end{align*}
which proves the second claim.
\end{proof}
\section{Proofs for the State-independent Case}
Before proving the proof we present some notations. Recall that,
\begin{align*}
    \phi(\theta,z)=\min_{y\in\Theta}\left(\left\langle z,y-\theta\right\rangle+\frac{\beta}{2}\|y-\theta\|_2^2\right),
\end{align*}
and
\begin{align*}
    y_k'=\argmin_{y\in\Theta}\left\lbrace\left\langle z_k,y-\theta_k\right\rangle+\frac{\beta}{2}\norm{y-\theta_k}_2^2\right\rbrace.
\end{align*}
For a given $\theta$, and $z$, we introduce the following notation for convenience.
\begin{align*}
    H(y)=\left\langle z,y-\theta\right\rangle+\frac{\beta}{2}\|y-\theta\|_2^2
\end{align*}
Noting that $H(y)$ is $\beta$-strongly convex, we have,
\begin{align*}
    \frac{\beta}{2}\|y_k-y_k'\|_2^2\leq H(y_k)-H(y_k').
\end{align*}
We choose the parameters of Algorithm~\ref{alg:icg} such that $$H(y_k)-H(y_k')\leq \delta_k^2.$$ The specific choice of $\delta_k$ will be set later. 
Let us define the following merit function.
\begin{align*}
    W(\theta,z)=(f(\theta)-f^*)-\phi(\theta,z)+\alpha_1\norm{\nabla f(\theta)-z}_2^2, \quad \alpha_1>0,\numberthis\label{eq:merithommc}
\end{align*}
where $f^* > -\infty$ is a uniform lower bound on the function $f$. We also need the following result from \cite{andrieu2005stability} on mixing properties of the data under Assumption~\ref{as:noise}~(a).
\begin{lemma}\cite{andrieu2005stability}\label{lm:expmixing}
Let Assumption~\ref{as:noise}~(a) be true. Then, for any $\theta\in\Theta$, the chain $\{x_k\}_k$ is exponentially mixing in the sense of Definition~\ref{eq:mixing}. 
\end{lemma}
\begin{proof}[Proof of Theorem~\ref{th:mainthmhomog}]
First we establish recursion relations on the three components of $W(\theta,z)$: $(f(\theta)-f^*)$, $\phi(\theta,z)$, and $\alpha_1\norm{\nabla f(\theta)-z}_2^2$. Using \eqref{eq:lipgrad}, Assumption~\ref{aspt:constraint}, Young's inequality,
\begin{align*}
   f(\theta_{k+1})-f(\theta_{k})\leq& \nabla f(\theta_k)^\top (\theta_{k+1}-\theta_k)+\frac{L_G}{2}\norm{\theta_{k+1}-\theta_k}_2^{2}\\
    =&\eta_{k+1}\nabla f(\theta_k)^\top (y_k'-\theta_{k})+\eta_{k+1}(\nabla f(\theta_k)-z_k)^\top (y_k-y_k')\\
    &+\eta_{k+1}(z_k+\beta (y_k'-\theta))^\top (y_k-y_k')-\eta_{k+1}\beta\inprod{y_k'-\theta_k,y_k-y_k'}+\frac{L_G\cD_\Theta^2\eta_{k+1}^2}{2}\\
    \leq& \eta_{k+1}\left(H(y_k)-H(y_k')-\frac{\beta}{2}\norm{y_k-y_k'}_2^2\right)+\frac{\eta_{k+1}\beta}{16}\norm{y_k'-\theta_k}_2^2+4\eta_{k+1}\beta\norm{y_k-y_k'}_2^2\\
    &+\frac{L_G\cD_\Theta^2\eta_{k+1}^2}{2}+\frac{\eta_{k+1}\beta}{16}\norm{\nabla f(\theta_k)-z_k}_2^2 +\frac{4\eta_{k+1}}{\beta}\norm{y_k-y_k'}_2^2+\eta_{k+1}\nabla f(\theta_k)^\top (y_k'-\theta_{k})\\
     \leq & \eta_{k+1}\left(H(y_k)-H(y_k')\right)+\frac{\eta_{k+1}\beta}{16}\norm{y_k'-\theta_k}_2^2+4\eta_{k+1}\beta\norm{y_k-y_k'}_2^2+\frac{L_G\cD_\Theta^2\eta_{k+1}^2}{2}\\
    &+\frac{\eta_{k+1}\beta}{16}\norm{\nabla f(\theta_k)-z_k}_2^2 +\frac{4\eta_{k+1}}{\beta}\norm{y_k-y_k'}_2^2+\eta_{k+1}\nabla f(\theta_k)^\top (y_k'-\theta_{k}).\numberthis\label{eq:fprogression}
\end{align*}
Using \eqref{eq:philipgrad},
\begin{align*}
    &\phi(\theta_k,z_k)-\phi(\theta_{k+1},z_{k+1})\\
    \leq& \inprod{z_k+\beta(y_k'-\theta_k),\theta_{k+1}-\theta_k}-\inprod{y_k'-\theta_k,z_{k+1}-z_k}+\frac{L_\phi}{2}\left[\norm{\theta_{k+1}-\theta_k}_2^2+\norm{z_{k+1}-z_k}_2^2\right]\\
    \leq& \eta_{k+1}\inprod{z_k+\beta(y_k'-\theta_k),y_k'-\theta_k}+\eta_{k+1}\inprod{z_k+\beta(y_k'-\theta_k),y_k-y_k'}-\inprod{y_k'-\theta_k,z_{k+1}-z_k}\\
    &+\frac{L_\phi}{2}\left[\norm{\theta_{k+1}-\theta_k}_2^2+\norm{z_{k+1}-z_k}_2^2\right]\\
    \leq& \eta_{k+1}\left(H(y_k)-H(y_k')-\frac{\beta}{2}\norm{y_k-y_k'}_2^2\right)-\eta_{k+1}\inprod{y_k'-\theta_k,\nabla F(\theta_k,x_{k+1})}\\
    &+\eta_{k+1}\inprod{y_k'-\theta_k,z_k}+\frac{L_\phi}{2}\left[\norm{\theta_{k+1}-\theta_k}_2^2+\norm{z_{k+1}-z_k}_2^2\right]\\
    \leq& -\eta_{k+1}\beta\norm{y_k'-\theta_k}_2^2+\eta_{k+1}\left(H(y_k)-H(y_k')\right)-\eta_{k+1}\inprod{y_k'-\theta_k,\nabla f(\theta_k)}\\
    &-\eta_{k+1}\inprod{y_k'-\theta_k,\xi_{k+1}(\theta_k,x_{k+1})}+\frac{L_\phi}{2}\left[\norm{\theta_{k+1}-\theta_k}_2^2+\norm{z_{k+1}-z_k}_2^2\right]
    .\numberthis\label{eq:phiprogression}
\end{align*}
Recall $\Gamma_i$ defined in \eqref{eq:gammadef}. Then
\begin{align*}
    \nabla f(\theta_{i})-z_{i}=&\frac{\Gamma_i}{\Gamma_{i-1}}\left(\nabla f(\theta_{i-1})-z_{i-1}\right)+\left(\nabla f(\theta_i)-\nabla f(\theta_{i-1})\right)-\eta_i\left(\tilde{\epsilon}_{i}+\nabla f(\theta_{i-1})-\nabla f(\theta_{i-1})\right)\\
    =&(1-\eta_{i})\left(\nabla f(\theta_{i-1})-z_{i-1}\right)+\frac{\eta_i}{\eta_i}\left(\nabla f(\theta_i)-\nabla f(\theta_{i-1})\right)-\eta_i\xi_{i}.
\end{align*}
Using Jensen's inequality, 
\begin{align*}
    \norm{\nabla f(\theta_{i})-z_{i}}_2^2\leq &(1-\eta_{i})\norm{\nabla f(\theta_{i-1})-z_{i-1}}_2^2+\frac{1}{\eta_i}\norm{\nabla f(\theta_i)-\nabla f(\theta_{i-1})}_2^2+\eta_i^2\norm{\xi_{i}}_2^2\\
    &-2\eta_i\inprod{\xi_i,(1-\eta_{i})\left(\nabla f(\theta_{i-1})-z_{i-1}\right)+\left(\nabla f(\theta_i)-\nabla f(\theta_{i-1})\right)}\\
    \leq &(1-\eta_{i})\norm{\nabla f(\theta_{i-1})-z_{i-1}}_2^2+2L_G^2\eta_i\norm{y_{i-1}'-\theta_{i-1}}_2^2+2L_G^2\eta_i\norm{y_{i-1}-y_{i-1}'}_2^2\\
    &+\eta_i^2\norm{\xi_{i}}_2^2-2\eta_i\inprod{\xi_i,(1-\eta_{i})\left(\nabla f(\theta_{i-1})-z_{i-1}\right)+\left(\nabla f(\theta_i)-\nabla f(\theta_{i-1})\right)}. \numberthis\label{eq:wcomp3recursion}
\end{align*}
Combining \eqref{eq:fprogression}, \eqref{eq:phiprogression}, and \eqref{eq:wcomp3recursion} we have,
\begin{align*}
    &W(\theta_{k+1},z_{k+1})-W(\theta_{k},z_{k})\\
    =&f(\theta_{k+1})-f(\theta_{k})-\phi(\theta_{k+1},z_{k+1})+\phi(\theta_k,z_k)+\alpha_1\norm{\nabla f(\theta_{k+1})-z_{k+1}}_2^2-\alpha_1\norm{\nabla f(\theta_{k})-z_{k}}_2^2\\
    \leq& 2\eta_{k+1}\left(H(y_k)-H(y_k')\right)-\frac{15\alpha_1\eta_{k+1}}{16}\norm{\nabla f(\theta_k)-z_k}_2^2-\left(\frac{15\beta\eta_{k+1}}{16}-2\alpha_1 L_G^2\eta_{k+1}\right)\norm{y_k'-\theta_k}_2^2\\
    &+\eta_{k+1}\left(4\beta+4/\alpha_1+2L_G^2\alpha_1\right)\|y_k-y_k'\|_2^2\\
    &+\eta_{k+1}^2\left(\frac{L_GD_\Theta^2}{2}+\frac{L_\phi D_\Theta^2}{2}+\norm{z_k-\nabla F(\theta_k,x_{k+1})}_2^2+\norm{\xi_{k+1}(\theta_k,x_{k+1})}_2^2+2\norm{\xi_{k+1}(\theta_k,x_{k+1})}_2\norm{\nabla f(\theta_{k})-z_k}_2\right)\\
    &-\eta_{k+1}\inprod{y_k'-\theta_k,\xi_{k+1}(\theta_k,x_{k+1})}-2\eta_{k+1}\inprod{\xi_{k+1}(\theta_k,x_{k+1}),\nabla f(\theta_{k+1})-z_k}.
\end{align*}
Rearranging, and choosing $\alpha_1=\beta/(32L_G^2)$ we get,
\begin{align*}
    &\frac{14\beta\eta_{k+1}}{16}\norm{y_k'-\theta_k}_2^2+\frac{15\beta\eta_{k+1}}{512L_G^2}\norm{\nabla f(\theta_k)-z_k}_2^2 \leq W(\theta_{k},z_{k})-W(\theta_{k+1},z_{k+1})\\&~~+\left(4/\beta+4\beta+4/\alpha_1+2L_G^2\alpha_1\right)\eta_{k+1}\delta_{k}^2
    +\eta_{k+1}^2U_{k}-\eta_{k+1}S_k-\eta_{k+1}Q_k, \numberthis\label{eq:finalboundintermed}
\end{align*}
where 
\begin{align*}
U_k&=\frac{L_GD_\Theta^2}{2}+\frac{L_\phi D_\Theta^2}{2}+\norm{z_k-\nabla F(\theta_k,x_{k+1})}_2^2+\norm{\xi_{k+1}(\theta_k,x_{k+1})}_2^2\\
&~~~~~~+2\norm{\xi_{k+1}(\theta_k,x_{k+1})}_2\norm{\nabla f(\theta_{k})-z_k}_2,\\ S_k&=\inprod{y_k'-\theta_k,\xi_{k+1}(\theta_k,x_{k+1})},\\
Q_k&=2\inprod{\xi_{k+1}(\theta_k,x_{k+1}),\nabla f(\theta_{k+1})-z_k}.
\end{align*}
Taking expectation on both sides and summing from $k=0$ to $k=N$, we get,
\begin{align*}
    &\sum_{k=0}^{N}\expec{\frac{14\beta\eta_{k+1}}{16}\norm{y_k'-\theta_k}_2^2+\frac{15\beta\eta_{k+1}}{512L_G^2}\norm{\nabla f(\theta_k)-z_k}_2^2}{}\leq W(\theta_{0},z_{0})\\
    +&\sum_{k=0}^N\left(4/\beta+4\beta+4/\alpha_1+2L_G^2\alpha_1\right)\eta_{k+1}\delta_{k}^2
    +\sum_{k=0}^N\eta_{k+1}^2\expec{U_{k}}{}-\sum_{k=0}^N\eta_{k+1}(\expec{S_k}{}+\expec{Q_k}{}),\numberthis\label{eq:hommcintermedbound}
\end{align*}
\textbf{Bound on $\expec{U_k}{}:$} Similar to \eqref{eq:expeczkbound}, we have $\expec{\norm{z_k}_2}{}\leq \sigma$. Using \L continuity of $f(\cdot)$, as explained in Section ~\ref{sec:assumption}, we have $\nabla f(\theta_k)\leq L$. Combining these with Assumption~\ref{as:finitenoisevar}, we have,
\begin{align*}
    \expec{U_k}{}=\mathcal{O}(1) \numberthis\label{eq:Ukbound}
\end{align*}
\textbf{Bound on $\expec{S_k}{}:$}
Using the fact that $\expec{\inprod{y_{k-l}'-\theta_{k-l},\xi_{k+1}(\theta_{k-l},x)}}{x\sim\pi}=0$, for $l\in\{1,\cdots,k-1\}$, we have
\begin{align*}
    \expec{S_k|\cF_{k-l}}{}=&\expec{\inprod{y_k'-\theta_k,\xi_{k+1}(\theta_k,x_{k+1})}|\cF_{k-l}}{}-\expec{\inprod{y_{k}'-\theta_{k},\xi_{k+1}(\theta_{k-l},x_{k+1})}|\cF_{k-l}}{}\\
    &+\expec{\inprod{y_k'-\theta_k,\xi_{k+1}(\theta_{k-l},x_{k+1})}|\cF_{k-l}}{}-\expec{\inprod{y_{k-l}'-\theta_{k-l},\xi_{k+1}(\theta_{k-l},x_{k+1})}|\cF_{k-l}}{}\\
    &+\expec{\inprod{y_{k-l}'-\theta_{k-l},\xi_{k+1}(\theta_{k-l},x_{k+1})}|\cF_{k-l}}{}-\expec{\inprod{y_{k-l}'-\theta_{k-l},\xi_{k+1}(\theta_{k-l},x)}}{x\sim\pi}\\
    =&\expec{\inprod{y_k'-\theta_k,\sum_{i=k-l+1}^k(\xi_{k+1}(\theta_i,x_{k+1})-\xi_{k+1}(\theta_{i-1},x_{k+1})}|\cF_{k-l}}{}\\
    &+\expec{\inprod{\sum_{i=k-l+1}^k(y_i'-\theta_i-y_{i-1}'+\theta_{i-1}),\xi_{k+1}(\theta_{k-l},x_{k+1})}|\cF_{k-l}}{}\\
    &+\inprod{y_{k-l}'-\theta_{k-l},\expec{\xi_{k+1}(\theta_{k-l},x_{k+1})|\cF_{k-l}}{}}-\inprod{y_{k-l}'-\theta_{k-l},\expec{\xi_{k+1}(\theta_{k-l},x)}{x\sim\pi}}\\
     =&\expec{\norm{y_k'-\theta_k}_2\sum_{i=k-l+1}^k\eta_i\norm{y_{i-1}-\theta_{i-1}}_2|\cF_{k-l}}{}\\
    &+\expec{\sum_{i=k-l+1}^k\left(\norm{z_i-z_{i-1}}_2/\beta+2\norm{\theta_i-\theta_{i-1}}_2\right)\norm{\xi_{k+1}(\theta_{k-l},x_{k+1})}_2|\cF_{k-l}}{}\\
    &+\norm{y_{k-l}'-\theta_{k-l}}_2\norm{\expec{\xi_{k+1}(\theta_{k-l},x_{k+1})|\cF_{k-l}}{}-\expec{\xi_{k+1}(\theta_{k-l},x)}{x\sim\pi}}_2.\numberthis\label{eq:Skintermedbound}
\end{align*}
Using Assumption~\ref{aspt:constraint} one has,
\begin{align*}
    \expec{\norm{y_k'-\theta_k}_2\sum_{i=k-l+1}^k\eta_i\norm{y_{i-1}-\theta_{i-1}}_2}{}=\mathcal{O}(l\eta_{k-l+1}).\numberthis\label{eq:bias1bound}
\end{align*}
Using Assumption~\ref{aspt:constraint}, Assumption~\ref{as:finitenoisevar},  $z_{k+1}-z_k=\eta_{k+1}(\nabla F(\theta_{k},x_{k+1})-z_k)$, and  $\expec{\|z_k\|_2}{}\leq \sigma$ one has that
\begin{align*}
    \expec{\sum_{i=k-l+1}^k\left(\norm{z_i-z_{i-1}}_2/\beta+2\norm{\theta_i-\theta_{i-1}}_2\right)\norm{\xi_{k+1}(\theta_{k-l},x_{k+1})}_2}{}=\mathcal{O}\left(l\eta_{k-l+1}\right).\numberthis\label{eq:bias2bound}
\end{align*}
Using Assumption~\ref{aspt:constraint}, Assumption~\ref{as:noise}, Lemma~\ref{lm:expmixing}, \eqref{eq:mixing}, and \L continuity of $f(\cdot)$, we have,
\begin{align*}
    \norm{y_{k-l}'-\theta_{k-l}}_2\norm{\expec{\xi_{k+1}(\theta_{k-l},x_{k+1})|\cF_{k-l}}{}-\expec{\xi_{k+1}(\theta_{k-l},x)}{x\sim\pi}}_2\leq \mathcal{O}(exp(-rl)), \numberthis\label{eq:mixingbias}
\end{align*}
where $r$ is as in ~\eqref{eq:mixing}.
Combining \eqref{eq:bias1bound}, \eqref{eq:bias2bound}, and \eqref{eq:mixingbias} with \eqref{eq:Skintermedbound} we get, 
\begin{align*}
    \expec{S_k}{}=\mathcal{O}(l\eta_{k-l+1}+exp(-rl)). \numberthis\label{eq:Skbound}
\end{align*}

\vspace{0.1in}
\noindent \textbf{Bound on $\expec{Q_k}{}$:} Following similar techniques used to establish bound on $\expec{S_k}{}$, we have,
\begin{align*}
    \expec{Q_k}{}=\mathcal{O}(l\eta_{k-l+1}+exp(-rl)) \numberthis\label{eq:Qkbound}
\end{align*}
Combining \eqref{eq:Ukbound}, \eqref{eq:Skbound}, and \eqref{eq:Qkbound} with \eqref{eq:finalboundintermed}, choosing $t_k=\ceil{\sqrt{k}}$ to ensure $\delta_k^2=\mathcal{\eta_k}$, setting $l=\ceil{\frac{\log (1/\eta_{k-l+1})}{r}}$, and choosing $\eta_k=(N+k)^{-a}$, $1/2<a<1$, we get,
\begin{align*}
    \sum_{k=0}^{N}\expec{\frac{14\beta\eta_{k+1}}{16}\norm{y_k'-\theta_k}_2^2+\frac{15\beta\eta_{k+1}}{512L_G^2}\norm{\nabla f(\theta_k)-z_k}_2^2}{}
    \leq W(\theta_{0},z_{0})+\mathcal{O}\left(N^{1-2a}\log N\right),\numberthis\label{eq:hommcintermedbound2}
\end{align*}
Dividing both sides by $\sum_{k=0}^{N}\eta_k$, and choosing $a=1/2$ we get,
\begin{align*}
    \expec{V(\theta_R,z_R)}{}=\mathcal{O}\left(\frac{\log N}{\sqrt{N}}\right).
\end{align*}
\end{proof}

\subsection*{Acknowledgements}
AR was affiliated with the Department of Statistics, UC Davis while this work was completed and was partially supported by National Science Foundation (NSF) grant CCF-1934568. KB was partially supported by a seed grant from the Center for Data Science and Artificial Intelligence Research, UC Davis and NSF Grant-2053918. SG was partially supported by an NSERC Discovery Grant. 
\bibliographystyle{alpha}
\bibliography{markov_asa}

\end{document}